\documentclass[a4paper]{amsart}

\usepackage{amsmath}
\usepackage{amssymb}
\usepackage{amsthm}
\usepackage{comment}
\usepackage{enumitem}
\usepackage{float}
\usepackage{mathtools}
\usepackage{tikz-cd}
\usepackage{stmaryrd}
\usepackage{xcolor}

\usepackage{tikz}
\usetikzlibrary{arrows.meta,decorations.markings,decorations.pathreplacing}

\usepackage[textsize=footnotesize]{todonotes}

\usepackage[hidelinks]{hyperref}

\newtheorem{thm}{Theorem}[section]
\newtheorem{lem}[thm]{Lemma}
\newtheorem{prop}[thm]{Proposition}
\newtheorem{cor}[thm]{Corollary}
\theoremstyle{definition}
\newtheorem{df}[thm]{Definition}
\newtheorem{ex}[thm]{Example}
\theoremstyle{remark}
\newtheorem{rmk}[thm]{Remark}

\makeatletter
\g@addto@macro\bfseries{\boldmath}
\makeatother

\numberwithin{equation}{section}

\DeclareMathOperator{\Coin}{Coin} 
\DeclareMathOperator{\coin}{coin} 
\DeclareMathOperator{\id}{id}	  
\DeclareMathOperator{\ind}{ind}   
\DeclareMathOperator{\MC}{MC}	  
\DeclareMathOperator{\pr}{pr}	  

\newcommand{\RR}{\mathbb{R}}      
\newcommand{\ZZ}{\mathbb{Z}}      

\newcommand{\orb}{\backslash}     
\newcommand{\R}{\mathcal{R}}      
\renewcommand{\S}{\Sigma}         

\newcommand{\eps}{\varepsilon}
\renewcommand{\iff}{\enskip\Leftrightarrow\enskip}
\renewcommand{\mod}{\;\mathrm{mod}\,}

\title{Nielsen coincidence theory for $(n,m)$-valued pairs}
\author{Karel Dekimpe and Lore De Weerdt}
\thanks{Research supported by Methusalem grant METH/21/03 -- long term structural funding of the Flemish Government.}
\thanks{K. Dekimpe was supported by a FWO sabbatical bench (K803525N)}
\address{KU Leuven Campus Kulak Kortrijk, 8500 Kortrijk, Belgium}
\email{Karel.Dekimpe@kuleuven.be}
\email{Lore.DeWeerdt@kuleuven.be}

\begin{document}

\vspace*{-2\baselineskip}

\maketitle

\vspace*{-1.25\baselineskip}

\begin{abstract}
We define a Nielsen coincidence number for pairs of maps $(f,g)$ where $f$ is $n$-valued and $g$ is $m$-valued. The classical definition of the Nielsen number does not work in this setting, but we propose a theory where coincidence classes are allowed to overlap. We invoke a notion of \emph{multiplicity} of a coincidence point, which indicates to how many classes it can belong, and we show that our Nielsen number is a homotopy-invariant lower bound for the number of coincidence points \emph{with multiplicity}. We show a computational result in terms of single-valued lift factors, and using this, a Wecken property for self-maps of the circle.
\end{abstract}

\section*{Introduction}

In \cite{brownkolahi}, the authors proposed a generalisation of Nielsen theory to coincidences of pairs $f,g:X\multimap Y$ where $f$ is an $n$-valued map and $g$ an $m$-valued map (for any $n,m\in \ZZ_{>0}$). However, their definition of a coincidence class, and the resulting definition of a Nielsen number, only really works if one of the maps is single-valued (see \cite{coin}). In general, the relation using which they want to partition the coincidence set into coincidence classes is not necessarily transitive. As a result, coincidence classes are no longer disjoint, and a lot of concepts from classical Nielsen theory (at first sight) seem to break down. For example, we will see below that using their definition of coincidence classes, distinct classes may combine into one through a homotopy (see Example \ref{ex:htpy}), which is not desired behaviour. Conversely, homotopies may take one coincidence class into two or more distinct classes, so there is no obvious way to define geometric essentiality (which is normally defined as: the class corresponding to the class in question under a homotopy is never empty). 

In this paper, we show that we can make a theory work where coincidence classes are allowed to overlap, starting from lifting classes instead of the geometric characterisation using homotopies of paths (used in \cite{brownkolahi}). The lifting classes themselves behave as usual; only their projections, the coincidence classes, can overlap. We show that the points belonging to more than one class can only be `multiple coincidences', i.e.\ points where the images of $f$ and $g$ have more than one point in common; and we define a geometric Nielsen number which is a homotopy-invariant lower bound for the number of coincidence points, \emph{counted with multiplicity}.

In order to actually compute Nielsen numbers, we need some notion of index. An index in the classical sense cannot be defined in this setting, but we establish a link with single-valued factors of lifts to a finite covering space (analogous to \cite{coin}), such that the index of a coincidence class can be defined in terms of the indices of coincidence classes of these single-valued lifts. Using this, we can explicitly compute Nielsen numbers on the circle, and show a Wecken property in that setting.

During the writing process of this paper, another paper \cite{graffstaeckerzeromski} appeared on arXiv which addresses the same question. They use a different approach, building a Nielsen theory for \emph{graph coincidences} (points $(x,y)\in X\times Y$ with $y\in f(x)\cap g(x)$). Defining an index seems a lot harder in that setting, and they only do it for self-maps of the circle. In that sense our theory is more general. We added an extra section at the end of the paper where we show that their Nielsen number (both the geometric one defined in general and the algebraic one defined using an index on the circle) is actually equal to ours, so the approaches are equivalent.

\bigskip

\section{Coincidence classes}

Let $X$ and $Y$ be connected finite polyhedra. An $n$-valued map $f:X\multimap Y$ is a continuous (i.e.\ upper- and lower-semicontinuous, see \cite{browngoncalves}) multifunction for which $|f(x)|=n$ for every $x\in X$. For an $n$-valued map $f:X\multimap Y$ and an $m$-valued map $g:X\multimap Y$, we define \[
\Coin(f,g)=\{x\in X \mid f(x)\cap g(x)\neq \varnothing \},
\]
following \cite{brownkolahi}.
We will define coincidence classes for the pair $(f,g)$. These will no longer constitute a partition of $\Coin(f,g)$: they might overlap, but we expect that the points belonging to multiple classes will be precisely the `multiple coincidences' of $f$ and $g$.

\begin{df}
We say $x\in \Coin(f,g)$ is a \emph{multiple coincidence} if $|f(x)\cap g(x)|>1$. We call $|f(x)\cap g(x)|$ the \emph{multiplicity} of the coincidence point $x$.
\end{df}

\subsection{Lifting classes}

Recall (see e.g.\ \cite{browngoncalves}) that an $n$-valued map $f:X\multimap Y$ can equivalently be viewed as a map $f:X\to D_n(Y)$, where \[
D_n(Y)=\{\{y_1,\ldots,y_n\}\subseteq Y\mid y_i\neq y_j \text{ if }i\neq j\}
\]
is the unordered configuration space, topologised as the quotient of the ordered configuration space \[
F_n(Y)=\{(y_1,\ldots,y_n)\in Y^n \mid y_i\neq y_j \text{ if }i\neq j\}
\]
under the action of the symmetric group $\Sigma_n$.

Using this configuration space viewpoint, a lifting theory for $n$-valued maps was described in \cite{charlotte1} (in the setting of self-maps, with the purpose of studying fixed points, but the concepts clearly generalise).
Let $p:\tilde{X}\to X$ be the universal cover of $X$, with covering group $\Pi\cong \pi_1(X)$, and $q:\tilde{Y}\to Y$ the universal cover of $Y$, with covering group $\Delta\cong \pi_1(Y)$. A covering space for $D_n(Y)$ is the \emph{orbit configuration space} \[
F_n(\tilde{Y},\Delta)=\{(\tilde{y}_1,\ldots,\tilde{y}_n)\in \tilde{Y}^n \mid q(\tilde{y}_i)\neq q(\tilde{y}_j) \text{ if }i\neq j \},
\] 
and one defines a \emph{lift} for $f$ as a map $\tilde{f}$ making the following diagram commute: \[
\begin{tikzcd}
\tilde{X} \ar[r,"\tilde{f}"] \ar[d,"p"'] & F_n(\tilde{Y},\Delta) \ar[d,"q^n"] \\
X \ar[r,"f"] & D_n(Y).
\end{tikzcd}
\]
Since $F_n(\tilde{Y},\Delta)\subseteq \tilde{Y}^n$, there are maps $\tilde{f}_i:\tilde{X}\to \tilde{Y}$ such that $\tilde{f}=(\tilde{f}_1,\ldots,\tilde{f}_n)$, so \[
f(p(\tilde{x}))=q^n(\tilde{f}(\tilde{x}))=\{q(\tilde{f}_1(\tilde{x})),\ldots,q(\tilde{f}_n(\tilde{x}))\}
\]
for all $\tilde{x}\in \tilde{X}$.
The maps $\tilde{f}_i$ are called \emph{lift factors} of $f$. It follows easily from single-valued lift theory that \begin{itemize}
\item if $\tilde{f}_i$ is a lift factor of $f$, then so is $\delta\tilde{f}_i\gamma$, for all $\gamma\in \Pi,\delta\in \Delta$;
\item given a fixed lift $\tilde{f}$ for $f$, any factor of another lift is of the form $\delta\tilde{f}_i$ for a unique $\delta\in \Delta$ and $i\in \{1,\ldots,n\}$;
\item if $\tilde{f}_i$ and $\tilde{f}'_j$ are two lift factors of $f$ such that $\tilde{f}_i(\tilde{x})=\tilde{f}'_j(\tilde{x})$ for some $\tilde{x}\in \tilde{X}$, then $\tilde{f}_i=\tilde{f}'_j$.
\end{itemize}

\begin{df}
Let $\tilde{f}=(\tilde{f}_1,\ldots,\tilde{f}_n)$ be a lift of $f$ and $\tilde{g}=(\tilde{g}_1,\ldots,\tilde{g}_m)$ a lift of $g$. For all $i$ and $j$, we call $(\tilde{f}_i,\tilde{g}_j)$ a \emph{lifting pair} of $(f,g)$.
\end{df}

\begin{lem}
We have \[
\Coin(f,g)=\bigcup_{(\tilde{f}_i,\tilde{g}_j)} p\Coin(\tilde{f}_i,\tilde{g}_j)
\]
where the union runs over all lifting pairs of $(f,g)$.
\end{lem}

\begin{proof}
If $\tilde{x}\in \Coin(\tilde{f}_i,\tilde{g}_j)$ for some lifting pair $(\tilde{f}_i,\tilde{g}_j)$, then \[
q(\tilde{f}_i(\tilde{x}))=q(\tilde{g}_j(\tilde{x}))\in f(p(\tilde{x}))\cap g(p(\tilde{x}))
\]
so $p(\tilde{x})\in \Coin(f,g)$.

Conversely, suppose $x\in \Coin(f,g)$, and pick $\tilde{x}\in p^{-1}(x)$. Fix lifts $\tilde{f}=(\tilde{f}_1,\ldots,\tilde{f}_n)$ for $f$ and $\tilde{g}=(\tilde{g}_1,\ldots,\tilde{g}_m)$ for $g$. Then \[
q^n(\tilde{f}(\tilde{x}))\cap q^m(\tilde{g}(\tilde{x}))=f(x)\cap g(x)\neq \varnothing,
\]
which means $q(\tilde{f}_i(\tilde{x}))=q(\tilde{g}_j(\tilde{x}))$ for some $i,j$. Hence there is an $\alpha\in \Delta$ such that $\alpha\tilde{f}_i(\tilde{x})=\tilde{g}_j(\tilde{x})$. Then $(\alpha\tilde{f}_i,\tilde{g}_j)$ is a lifting pair for $(f,g)$ with $x\in p\Coin(\alpha\tilde{f}_i,\tilde{g}_j)$.
\end{proof}

\begin{df}
We call two lifting pairs $(\tilde{f}_i,\tilde{g}_j)$ and $(\tilde{f}'_k,\tilde{g}'_\ell)$ \emph{equivalent}, and write $(\tilde{f}_i,\tilde{g}_j)\sim(\tilde{f}'_k,\tilde{g}'_\ell)$, if there exist $\gamma\in \Pi$ and $\delta\in \Delta$ such that $\tilde{f}'_k=\delta\tilde{f}_i\gamma$ and $\tilde{g}'_\ell=\delta\tilde{g}_j\gamma$.
\end{df}

This clearly defines an equivalence relation on the set of lifting pairs. We call the equivalence classes the \emph{lifting classes} of $(f,g)$, and the number of lifting classes the \emph{Reidemeister number} $R(f,g)$.

\begin{prop}\label{prop:equiv-lifts}
Let $(\tilde{f}_i,\tilde{g}_j)$ and $(\tilde{f}'_k,\tilde{g}'_\ell)$ be two lifting pairs of $(f,g)$.
\begin{itemize}
\item[\rm(i)] If $(\tilde{f}_i,\tilde{g}_j)\sim(\tilde{f}'_k,\tilde{g}'_\ell)$, then $p\Coin(\tilde{f}_i,\tilde{g}_j)= p\Coin(\tilde{f}'_k,\tilde{g}'_\ell)$. 
\item[\rm(ii)] If $(\tilde{f}_i,\tilde{g}_j)\nsim(\tilde{f}'_k,\tilde{g}'_\ell)$, then $p\Coin(\tilde{f}_i,\tilde{g}_j)\cap p\Coin(\tilde{f}'_k,\tilde{g}'_\ell)$ can only contain multiple coincidence points.
\end{itemize}
\end{prop}

\begin{proof}
For (i), suppose $\tilde{f}'_k=\delta\tilde{f}_i\gamma$ and $\tilde{g}'_\ell=\delta\tilde{g}_j\gamma$, and pick $x\in p\Coin(\tilde{f}_i,\tilde{g}_j)$. Take $\tilde{x}\in p^{-1}(x)$ with $\tilde{f}_i(\tilde{x})=\tilde{g}_j(\tilde{x})$. Then \[
\begin{split}
\tilde{f}'_k(\gamma^{-1}\tilde{x})=\delta\tilde{f}_i\gamma(\gamma^{-1}\tilde{x})&=\delta\tilde{f}_i(\tilde{x})\\&=\delta\tilde{g}_j(\tilde{x})=\delta\tilde{g}_j\gamma(\gamma^{-1}\tilde{x})=\tilde{g}'_\ell(\gamma^{-1}\tilde{x}),
\end{split}
\]
so $\gamma^{-1}\tilde{x}\in p^{-1}(x)$ belongs to $\Coin(\tilde{f}'_k,\tilde{g}'_\ell)$; that is, $x\in p\Coin(\tilde{f}'_k,\tilde{g}'_\ell)$. The converse inclusion is analogous.

For (ii), suppose $x\in p\Coin(\tilde{f}_i,\tilde{g}_j)\cap p\Coin(\tilde{f}'_k,\tilde{g}'_\ell)$ and $|f(x)\cap g(x)|=1$; we will show that $(\tilde{f}_i,\tilde{g}_j)\sim(\tilde{f}'_k,\tilde{g}'_\ell)$.

Pick $\tilde{x}\in p^{-1}(x)$ with $\tilde{f}_i(\tilde{x})=\tilde{g}_j(\tilde{x})$ and $\tilde{x}'\in p^{-1}(x)$ with $\tilde{f}'_k(\tilde{x}')=\tilde{g}'_\ell(\tilde{x}')$. Write $\tilde{x}'=\gamma\tilde{x}$ with $\gamma\in \Pi$. Then we have \[
\begin{split}
q(\tilde{f}_i(\tilde{x}))&=q(\tilde{g}_j(\tilde{x}))\in f(x)\cap g(x) \\
q(\tilde{f}'_k(\gamma\tilde{x}))&=q(\tilde{g}'_\ell(\gamma\tilde{x}))\in f(x)\cap g(x),
\end{split}
\]
so since $|f(x)\cap g(x)|=1$, \[
q(\tilde{f}_i(\tilde{x}))=q(\tilde{f}'_k(\gamma\tilde{x})).
\]
It follows that there is a $\delta\in \Delta$ such that \[
\tilde{f}_i(\tilde{x})=\delta\tilde{f}'_k(\gamma\tilde{x})=\delta\tilde{f}'_k\gamma(\tilde{x}).
\]
Since lift factors are uniquely determined by their image of a point, it follows that $\tilde{f}_i=\delta\tilde{f}'_k\gamma$. We also have \[
\tilde{g}_j(\tilde{x})=\tilde{f}_i(\tilde{x})=\delta\tilde{f}'_k(\gamma\tilde{x})=\delta\tilde{g}'_\ell(\gamma\tilde{x})=\delta\tilde{g}'_\ell\gamma(\tilde{x}),
\]
so in the same way, $\tilde{g}_j=\delta\tilde{g}'_\ell\gamma$.
\end{proof}

Thus, we can write \[
\Coin(f,g)=\bigcup_{[(\tilde{f}_i,\tilde{g}_j)]} p\Coin(\tilde{f}_i,\tilde{g}_j)
\]
where the union runs over all lifting classes of $(f,g)$. We call the sets $p\Coin(\tilde{f}_i,\tilde{g}_j)$, indexed by the lifting classes, the \emph{coincidence classes} of $(f,g)$. Some of these classes may be empty, and classes may overlap, or even be equal, but they are considered distinct when labeled by distinct lifting classes. The points that belong to multiple classes can only be multiple coincidence points.

\subsection{Homotopies of paths}

It is well-known (see e.g.\ \cite{browngoncalves}) that any $n$-valued map defined on a simply connected space \emph{splits}, i.e., its image is the union of the images of $n$ single-valued maps. In particular, given a path $c:I\to X$, the $n$-valued map $fc:I\multimap Y$ splits into $n$ paths in $Y$. Explicitly, if $\tilde{f}=(\tilde{f}_1,\ldots,\tilde{f}_n):\tilde{X}\to F_n(\tilde{Y},\Delta)$ is a lift of $f$, and $\tilde{c}:I\to \tilde{X}$ is a lift of $c$, then \[
fc=q^n\tilde{f}\tilde{c}=\{q\tilde{f}_1\tilde{c},\ldots,q\tilde{f}_n\tilde{c}\},
\]
so the factors of $fc$ are precisely the paths $q\tilde{f}_i\tilde{c}:I\to Y$. Similarly for $g$.

\begin{prop}\label{prop:htpy}
Two points $x,y\in \Coin(f,g)$ belong to the same coincidence class if and only if there is a path $c:I\to X$ from $x$ to $y$ such that $fc$ and $gc$ have a path homotopic factor.
\end{prop}

\begin{proof}
Whenever we talk about a homotopy between paths in this proof, we mean a \emph{path homotopy}, i.e.\ a homotopy relative to the endpoints.
	
Suppose $x,y\in \Coin(f,g)$ belong to the same coincidence class $p\Coin(\tilde{f}_i,\tilde{g}_j)$, for some lifting pair $(\tilde{f}_i,\tilde{g}_j)$ of $(f,g)$; say $\tilde{x}\in p^{-1}(x)$ and $\tilde{y}\in p^{-1}(y)$ belong to $\Coin(\tilde{f}_i,\tilde{g}_j)$. Let $\tilde{c}:I\to \tilde{X}$ be a path from $\tilde{x}$ to $\tilde{y}$. Since \[
\begin{split}
\tilde{f}_i\tilde{c}(0)=\tilde{f}_i(\tilde{x})&=\tilde{g}_j(\tilde{x})=\tilde{g}_j\tilde{c}(0) \\
\tilde{f}_i\tilde{c}(1)=\tilde{f}_i(\tilde{y})&=\tilde{g}_j(\tilde{y})=\tilde{g}_j\tilde{c}(1)
\end{split}
\] 
and $\tilde{Y}$ is simply connected, the paths $\tilde{f}_i\tilde{c}$ and $\tilde{g}_j\tilde{c}$ are homotopic relative to the endpoints. Therefore, the paths $q\tilde{f}_i\tilde{c}$ and $q\tilde{g}_j\tilde{c}$ are homotopic in $Y$, so these are homotopic factors of $fc$ and $gc$.

Conversely, for $x,y\in \Coin(f,g)$, suppose $c:I\to X$ is a path from $x$ to $y$ such that $fc$ and $gc$ have a homotopic factor. The homotopic factors are of the form $q\tilde{f}_i\tilde{c}$ and $q\tilde{g}_j\tilde{c}$, where $\tilde{c}$ is a lift for $c$, and $\tilde{f}_i$ and $\tilde{g}_j$ are lift factors of $f$ and $g$, respectively. Since $q\tilde{f}_i\tilde{c}(0)=q\tilde{g}_j\tilde{c}(0)$, there is an $\alpha\in \Delta$ such that $\alpha\tilde{f}_i\tilde{c}(0)=\tilde{g}_j\tilde{c}(0)$. Then $\alpha\tilde{f}_i\tilde{c}$ and $\tilde{g}_j\tilde{c}$ are paths in $\tilde{Y}$ with homotopic projections and with the same beginpoint, so they must also have the same endpoint: $\alpha\tilde{f}_i\tilde{c}(1)=\tilde{g}_j\tilde{c}(1)$. Thus, both $\tilde{c}(0)\in p^{-1}(x)$ and $\tilde{c}(1)\in p^{-1}(y)$ belong to $\Coin(\alpha\tilde{f}_i,\tilde{g}_j)$, hence $x$ and $y$ belong to the same coincidence class $p\Coin(\alpha\tilde{f}_i,\tilde{g}_j)$.
\end{proof}

This is precisely how coincidence classes were defined in \cite{brownkolahi}. In our framework, this is indeed an equivalent characterisation, but it cannot be used as a \emph{definition} for coincidence classes: for example, when two coincidence classes are equal as a set, this definition would only detect one class.

\begin{ex}\label{ex:htpy}
Consider the maps $f$ and $g$ on the circle whose graphs are shown in Figure \ref{fig:coin-nm-hom} (where the circle is represented as the interval with the endpoints identified).
The pair $(f,g)$ has two coincidence points, which clearly belong to different classes. But after a homotopy $f\simeq f'$, these points combine into a single coincidence point. 

If we would define coincidence classes using homotopies of paths, the pair $(f',g)$ would only have a single coincidence class. So it would look like two coincidence classes can be combined into one by a homotopy. However, with the definition using lifting classes, one would see that $(f',g)$ still has two coincidence classes, both equal as a set to the single coincidence point, but labeled by different lifting classes.

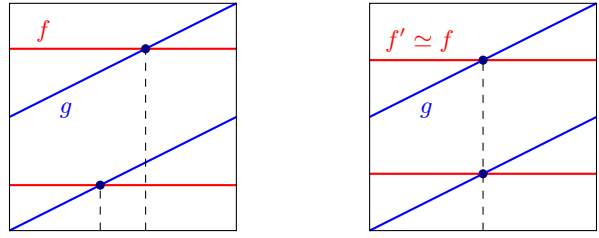
\begin{figure}[h!]
\centering
\medskip
\begin{tikzpicture}[scale=3]
{\draw[thick,red] (0,0.2)--(1,0.2);
\draw[thick,red] (0,0.8)--(1,0.8);}
\draw[thick,blue] (0,0)--(1,0.5);
\draw[thick,blue] (0,0.5)--(1,1);
{\draw[red] (0.15,0.88) node{\small$f$};}
\draw[blue] (0.25,0.53) node{\small$g$};
{\draw[dashed] (0.4,0)--(0.4,0.2);
\draw[dashed] (0.6,0)--(0.6,0.8);
\filldraw[blue!50!black] (0.4,0.2) circle (0.5pt);
\filldraw[blue!50!black] (0.6,0.8) circle (0.5pt);}
\draw (0,0) rectangle (1,1);
\end{tikzpicture}
\hspace{1.5cm}
\begin{tikzpicture}[scale=3]
{\draw[thick,red] (0,0.25)--(1,0.25);
\draw[thick,red] (0,0.75)--(1,0.75);}
\draw[thick,blue] (0,0)--(1,0.5);
\draw[thick,blue] (0,0.5)--(1,1);
{\draw[red] (0.22,0.84) node{\small$f'\simeq f$};}
\draw[blue] (0.25,0.53) node{\small$g$};
{\draw[dashed] (0.5,0)--(0.5,0.75);
\filldraw[blue!50!black] (0.5,0.25) circle (0.5pt);
\filldraw[blue!50!black] (0.5,0.75) circle (0.5pt);}
\draw (0,0) rectangle (1,1);
\end{tikzpicture}
\caption{Defining coincidence classes using homotopies of paths would not lead to a well-behaved notion of coincidence classes, since it cannot detect multiple coincidence classes that are equal as sets.}
\label{fig:coin-nm-hom}
\end{figure}
\end{ex}

\subsection{Reidemeister classes}

Equivalence of lifting pairs can be checked in an algebraic way as follows. Fix lifts $\tilde{f}=(\tilde{f}_1,\ldots,\tilde{f}_n)$ for $f$ and $\tilde{g}=(\tilde{g}_1,\ldots,\tilde{g}_m)$ for $g$. These lifts induce morphisms \[
\begin{split}
\varphi:\Pi&\to \Delta^n\rtimes \S_n:\gamma\mapsto (\varphi_1(\gamma),\ldots,\varphi_n(\gamma);\sigma_\gamma) \\
\psi:\Pi&\to \Delta^m\rtimes \S_m:\gamma\mapsto (\psi_1(\gamma),\ldots,\psi_n(\gamma);\tau_\gamma)
\end{split}
\]
such that for all $\gamma\in \Pi$, $i\in \{1,\ldots,n\}$ and $j\in \{1,\ldots,m\}$, \[
\begin{split}
\tilde{f}_i\gamma&=\varphi_i(\gamma)\tilde{f}_{\sigma^{-1}_\gamma(i)} \\
\tilde{g}_j\gamma&=\psi_j(\gamma)\tilde{g}_{\tau^{-1}_\gamma(j)}
\end{split}
\]
(see \cite{charlotte1,coin}).
Note that the maps $\sigma:\Pi\to \S_n$ and $\tau:\Pi\to \S_m$ are group morphisms, and the maps $\varphi_i:\Pi\to \Delta$ and $\psi_j:\Pi\to \Delta$ satisfy \begin{equation}\label{eq:phi-morf}
\begin{split}
\forall \gamma,\gamma'\in\Pi:\enskip\varphi_i(\gamma\gamma')&=\varphi_i(\gamma)\varphi_{\sigma_\gamma^{-1}(i)}(\gamma') \\ \psi_j(\gamma\gamma')&=\psi_j(\gamma)\psi_{\tau_\gamma^{-1}(j)}(\gamma').
\end{split}
\end{equation}

\begin{lem}
Any lifting class for $(f,g)$ contains a pair $(\alpha\tilde{f}_i,\tilde{g}_j)$, for some $\alpha\in \Delta$, $i\in \{1,\ldots,n\}$ and $j\in \{1,\ldots,m\}$.
\end{lem}

\begin{proof}
Any lifting pair of $(f,g)$ is of the form $(\alpha\tilde{f}_i,\beta\tilde{g}_j)$; this pair is equivalent to $(\beta^{-1}\alpha\tilde{f}_i,\tilde{g}_j)$.
\end{proof}

\begin{prop}\label{prop:reidemeister1}
Two lifting pairs $(\alpha\tilde{f}_i,\tilde{g}_j)$ and $(\beta\tilde{f}_k,\tilde{g}_\ell)$ are equivalent if and only if there exists a $\gamma\in \Pi$ such that \[
\begin{cases}
i=\sigma^{-1}_\gamma(k) \text{ and } j=\tau^{-1}_\gamma(\ell), \\
\alpha=\psi_\ell(\gamma)^{-1}\beta\varphi_k(\gamma).
\end{cases}
\]
\end{prop}

\begin{proof}
Suppose $(\alpha\tilde{f}_i,\tilde{g}_j)\sim (\beta\tilde{f}_k,\tilde{g}_\ell)$. Then there are $\gamma\in \Pi$, $\delta\in \Delta$ such that \[
\begin{cases}
\alpha\tilde{f}_i=\delta\beta\tilde{f}_k\gamma=\delta\beta\varphi_k(\gamma)\tilde{f}_{\sigma^{-1}_\gamma(k)} \\
\tilde{g}_j=\delta\tilde{g}_\ell\gamma=\delta\psi_\ell(\gamma)\tilde{g}_{\tau^{-1}_\gamma(\ell)}.
\end{cases}
\]
It follows that $i=\sigma^{-1}_\gamma(k)$ and $\alpha=\delta\beta\varphi_k(\gamma)$, and that $j=\tau^{-1}_\gamma(\ell)$ and $1=\delta\psi_\ell(\gamma)$. From the last equation, $\delta=\psi_\ell(\gamma)^{-1}$, so that $\alpha=\psi_\ell(\gamma)^{-1}\beta\varphi_k(\gamma)$.

Conversely, if there exists a $\gamma\in \Pi$ such that $i=\sigma^{-1}_\gamma(k)$, $j=\tau^{-1}_\gamma(\ell)$ and $\alpha=\psi_\ell(\gamma)^{-1}\beta\varphi_k(\gamma)$, then \[
\begin{gathered}
\alpha\tilde{f}_i=\psi_\ell(\gamma)^{-1}\beta\varphi_k(\gamma)\tilde{f}_{\sigma^{-1}_\gamma(k)}=\psi_\ell(\gamma)^{-1}\beta\tilde{f}_k\gamma
\\
\tilde{g}_j=\tilde{g}_{\tau^{-1}_\gamma(\ell)}=\psi_\ell(\gamma)^{-1}\psi_\ell(\gamma)\tilde{g}_{\tau^{-1}_\gamma(\ell)}=\psi_\ell(\gamma)^{-1}\tilde{g}_\ell\gamma
\end{gathered}
\]
so $(\alpha\tilde{f}_i,\tilde{g}_j)\sim (\beta\tilde{f}_k,\tilde{g}_\ell)$.
\end{proof}

By \eqref{eq:phi-morf}, the restrictions $\varphi_i|_{\{\gamma\in \Pi \mid \sigma_\gamma(i)=i \}}$ and $\psi_j|_{\{\gamma\in \Pi \mid \tau_\gamma(j)=j \}}$ are group morphisms. In particular, so are the restrictions $\varphi_i|_{S_{ij}}$ and $\psi_j|_{S_{ij}}$, where \[
S_{ij}\vcentcolon=\{\gamma\in \Pi \mid \sigma_\gamma(i)=i \text{ and } \tau_\gamma(j)=j \}.
\]
We will express the equivalence of lifting pairs in terms of these morphisms.

We define an equivalence relation $\sim_{\sigma,\tau}$ on $\{1,\ldots,n\}\times\{1,\ldots,m\}$ by \[
(i,j) \sim_{\sigma,\tau} (k,\ell) \iff \exists \gamma\in \Pi: k=\sigma_\gamma(i) \text{ and } \ell=\tau_\gamma(j)
\]
and, for all $i,j$, an equivalence relation $\sim_{\varphi_i,\psi_j}$ on $\Delta$ by \[
\alpha \sim_{\varphi_i,\psi_j} \beta \iff \exists \gamma\in S_{ij}:\alpha=\psi_j(\gamma^{-1})\beta\varphi_i(\gamma).
\]

\begin{prop}
If $(i,j) \nsim_{\sigma,\tau} (k,\ell)$, then $(\alpha\tilde{f}_i,\tilde{g}_j)\nsim (\beta\tilde{f}_k,\tilde{g}_\ell)$ for all $\alpha,\beta\in \Delta$. If $(i,j) \sim_{\sigma,\tau} (k,\ell)$, then \[
\{[(\alpha\tilde{f}_i,\tilde{g}_j)] \mid \alpha\in \Delta \} = \{[(\alpha\tilde{f}_k,\tilde{g}_\ell)] \mid \alpha\in \Delta \},
\]
and for a fixed pair $(i,j)$, \[
(\alpha\tilde{f}_i,\tilde{g}_j)\sim (\beta\tilde{f}_i,\tilde{g}_j) \iff \alpha\sim_{\varphi_i,\psi_j} \beta.
\]
\end{prop}

\begin{proof}
The first statement is immediate from Proposition \ref{prop:reidemeister1}.

Suppose $(i,j) \sim_{\sigma,\tau} (k,\ell)$, and take $\gamma\in \Pi$ such that $k=\sigma_\gamma(i)$ and $\ell=\tau_\gamma(j)$. Then, by Proposition \ref{prop:reidemeister1}, \[
(\alpha\tilde{f}_i,\tilde{g}_j)\sim (\psi_\ell(\gamma)\alpha\varphi_k(\gamma)^{-1}\tilde{f}_k,\tilde{g}_\ell),
\]
so $(\alpha\tilde{f}_i,\tilde{g}_j)\in \{[(\beta\tilde{f}_k,\tilde{g}_\ell)] \mid \beta\in \Delta \}$. The converse inclusion follows by symmetry. 

For a fixed pair $(i,j)$, Proposition \ref{prop:reidemeister1} states that \[
(\alpha\tilde{f}_i,\tilde{g}_j)\sim (\beta\tilde{f}_i,\tilde{g}_j) \iff \exists \gamma\in \Pi: \begin{cases}
i=\sigma_\gamma(i) \text{ and } j=\tau_\gamma(j) \\
\alpha=\psi_j(\gamma)^{-1}\beta\varphi_i(\gamma).
\end{cases}
\]
The first condition means precisely that $\gamma\in S_{ij}$, and since $\psi_j|_{S_{ij}}$ is a group morphism, we have $\psi_j(\gamma)^{-1}=\psi_j(\gamma^{-1})$.
\end{proof}

We write $\R[\varphi_i,\psi_j]=\Delta/\sim_{\varphi_i,\psi_j}$, and call this the set of \emph{Reidemeister classes} of the pair of morphisms $\varphi_i,\psi_j:S_{ij}\to \Delta$.

To conclude, we can write the set of all distinct coincidence classes of $(f,g)$ as \[
\{p\Coin(\alpha\tilde{f}_i,\tilde{g}_j) \mid [(i,j)]_{\sigma,\tau}; [\alpha]\in \R[\varphi_i,\psi_j] \}
\]
where $[(i,j)]_{\sigma,\tau}$ is short for $[(i,j)]\in (\{1,\ldots,n\}\times\{1,\ldots,m\})/\sim_{\sigma,\tau}$.

\begin{ex}\label{ex:coin}
Let us revisit the example we gave in \cite{coin} to illustrate that the coincidence relation of \cite{brownkolahi} is not transitive: consider the pair of $2$-valued maps of the circle whose graphs are drawn in Figure \ref{fig:ex:coin}. 
Explicitly: \[
\begin{split}
f&: \RR/\ZZ \multimap \RR/\ZZ: t+\ZZ\mapsto \{\tilde{f}_1(t)+\ZZ,\tilde{f}_2(t)+\ZZ \} \\
g&: \RR/\ZZ \multimap \RR/\ZZ: t+\ZZ\mapsto \{\tilde{g}_1(t)+\ZZ,\tilde{g}_2(t)+\ZZ \}
\end{split}
\]
where $\tilde{f}_1(t)=\frac{1}{4}$ and $\tilde{f}_2(t)=\frac{3}{4}$ for all $t$, and $\tilde{g}_1,\tilde{g}_2$ are defined on $[0,1]$ by \[
\tilde{g}_1(t)=\begin{cases}
	t & t\in [0,\frac{1}{4}] \\
	\frac{1}{4} & t\in [\frac{1}{4},\frac{1}{2}] \\
	\frac{t}{2} & t\in [\frac{1}{2},1]
\end{cases}
\qquad\qquad
\tilde{g}_2(t)=\begin{cases}
	\frac{t+1}{2} & t\in [0,\frac{1}{2}] \\
	\frac{3}{4} & t\in [\frac{1}{2},\frac{3}{4}] \\
	t & t\in [\frac{3}{4},1]
\end{cases}
\]
and extended to $\RR$ by setting $\tilde{g}_1(t+2z)=\tilde{g}_1(t)+z$, $\tilde{g}_1(t+2z+1)=\tilde{g}_2(t)+z$, $\tilde{g}_2(t+2z)=\tilde{g}_2(t)+z$ and $\tilde{g}_2(t+2z+1)=\tilde{g}_1(t)+z+1$ for all $t\in [0,1]$, $z\in \ZZ$.

\begin{figure}[h!]
\begin{tikzpicture}[scale=3]
	\draw (0,0) rectangle (1,1);
	\draw[thick,red] (0,0.25)--(1,0.25);
	\draw[thick,red] (0,0.75)--(1,0.75);
	\draw[thick,blue] (0,0)--(0.25,0.25)--(0.5,0.25)--(1,0.5);
	\draw[thick,blue] (0,0.5)--(0.5,0.75)--(0.75,0.75)--(1,1);
	\draw[red] (0.15,0.85) node{\small$f$};
	\draw[blue] (0.2,0.5) node{\small$g$};
	
	\draw[dashed] (0.25,0)node[below]{\small$\frac{1}{4}$}--(0.25,0.25);
	\draw[dashed] (0.5,0)node[below]{\small$\frac{1}{2}$}--(0.5,0.75);
	\draw[dashed] (0.75,0)node[below]{\small$\frac{3}{4}$}--(0.75,0.75);
	
	\draw[dashed] (0,0.25)node[left]{\small$\frac{1}{4}$};
	\draw[dashed] (0,0.5)node[left]{\small$\frac{1}{2}$};
	\draw[dashed] (0,0.75)node[left]{\small$\frac{3}{4}$};
\end{tikzpicture}
\caption{A pair of $2$-valued maps of the circle with intersecting coincidence classes.}
\label{fig:ex:coin}
\end{figure}
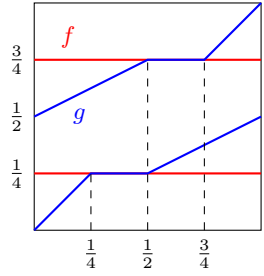
\end{ex}

In \cite{coin}, we argued geometrically that there are two coincidence classes and the double coincidence point $\frac{1}{2}+\ZZ$ belongs to both of them. We can now show this using the algebraic theory developed above.

One easily verifies that the morphisms $\varphi:\ZZ\to \ZZ^2\rtimes \S_2$ and $\psi:\ZZ\to \ZZ^2\rtimes \S_2$ induced by the lifts $(\tilde{f}_1,\tilde{f}_2)$ and $(\tilde{g}_1,\tilde{g}_2)$ are given by $\varphi(1)=(0,0;\id)$ and $\psi(1)=(0,1;(12))$ (where $(12)\in \S_2$ is the permutation that interchanges $1$ and $2$). It follows that $S_{ij}=2\ZZ$ for all $i,j\in \{1,2\}$, and that $\varphi_1(2z)=\varphi_2(2z)=0$ and $\psi_1(2z)=\psi_2(2z)=z$ for all $z\in \ZZ$.

Hence, there are two equivalence classes $\{(1,1),(1,2)\}$ and $\{(2,1),(2,2)\}$ for the relation $\sim_{\sigma,\tau}$. Let us take $(1,1)$ and $(2,2)$ as representatives. For $v,w\in \ZZ$, we have \[
\begin{split}
v\sim_{\varphi_1,\psi_1} w &\iff \exists z\in \ZZ: v=w+\varphi_1(2z)-\psi_1(2z)=w-z \\
v\sim_{\varphi_2,\psi_2} w &\iff \exists z\in \ZZ: v=w+\varphi_2(2z)-\psi_2(2z)=w-z.
\end{split}
\]
For both relations, there is only one equivalence class.
Hence there are two lifting classes \[
\begin{split}
[(\tilde{f}_1,\tilde{g}_1)]&=\{(\tilde{f}_1+z,\tilde{g}_i) \mid z\in \ZZ, i\in \{1,2\} \} \\
[(\tilde{f}_2,\tilde{g}_2)]&=\{(\tilde{f}_2+z,\tilde{g}_i) \mid z\in \ZZ, i\in \{1,2\} \}.
\end{split}
\]
The corresponding coincidence classes are \[
\begin{split}
p\Coin(\tilde{f}_1,\tilde{g}_1)&=\{t+\ZZ \mid t\in [\tfrac{1}{4},\tfrac{1}{2}]\} \\
p\Coin(\tilde{f}_2,\tilde{g}_2)&=\{t+\ZZ \mid t\in [\tfrac{1}{2},\tfrac{3}{4}]\},
\end{split}
\]
consistent with our intuition from \cite{coin}.

\subsection{Partitions}\label{subsec:partitions}

Recall from \cite{staecker2021} that a \emph{partition} of $n$-valued map $f:X\multimap Y$ is a set of multimaps $\{f_1,\ldots,f_k\}$, where each $f_i$ is $n_i$-valued with $n_1+\ldots+n_k=n$, such that $f(x)=f_1(x)\cup\ldots\cup f_k(x)$ for all $x\in X$. An $n$-valued map $f$ is called \emph{irreducible} if there exists no partition of $f$ into smaller-valued maps. In analogy with the results of \cite{staecker2021} for fixed points, we have:

\begin{lem}
If $f=\{f_1,f_2\}$ is a partition, then \[
R(f,g)=R(f_1,g)+R(f_2,g).
\]
\end{lem}

\begin{proof}
We show that the set of lifting classes of $(f,g)$ is the disjoint union of the sets of lifting classes of $(f_1,g)$ and of $(f_2,g)$.

Fix a lift $\tilde{f}=(\tilde{f}_1,\ldots,\tilde{f}_n)$ for $f$. We may reorder the factors of $\tilde{f}$ so that $(\tilde{f}_1,\ldots,\tilde{f}_{n_1})$ is a lift for $f_1$ and $(\tilde{f}_{n_1+1},\ldots,\tilde{f}_n)$ is a lift for $f_2$ (with $n_1$ such that $f_1$ is $n_1$-valued and $f_2$ is $(n-n_1)$-valued). Then the set of lifting classes of $(f_1,g)$ is \[
\{[(\alpha\tilde{f}_i,\tilde{g}_j)] \mid \alpha\in \Delta, i\in \{1,\ldots,n_1\},j\in\{1,\ldots,m\} \}
\]
and the set of lifting classes of $(f_2,g)$ is \[
\{[(\alpha\tilde{f}_i,\tilde{g}_j)] \mid \alpha\in \Delta, i\in \{n_1+1,\ldots,n\},j\in\{1,\ldots,m\} \}.
\]
Clearly, the union of these sets is the set of lifting classes of $(f,g)$. To show they are disjoint, let $\varphi=(\varphi_1,\ldots,\varphi_n;\sigma)$ be the morphism induced by $\tilde{f}$. It is shown in \cite{staecker2021} that for $i\in \{1,\ldots,n_1\}$ and $k\in \{n_1+1,\ldots,n\}$, there is no $\gamma\in \Pi$ such that $\sigma_\gamma(i)=k$. Hence for such $i$ and $k$, the lifting pairs $(\alpha\tilde{f}_i,\tilde{g}_j)$ and $(\beta\tilde{f}_k,\tilde{g}_\ell)$ are inequivalent for all $j,\ell$, $\alpha$ and $\beta$.
\end{proof}

\begin{cor}
If $f=\{f_1,\ldots,f_k\}$ and $g=\{g_1,\ldots,g_\ell\}$ are the partitions into irreducibles, then \[
R(f,g)=\sum_{i=1}^k \sum_{j=1}^\ell R(f_i,g_j).
\]
\end{cor}

\subsection{Multiple coincidences}

We saw in Proposition \ref{prop:equiv-lifts} that a point $x$ for which $f(x)\cap g(x)$ is a singleton can only belong to one coincidence class. In general, we can relate the number of coincidence classes containing $x$ to the number of elements in $f(x)\cap g(x)$, i.e., the multiplicity of $x$. 

\begin{lem}
Fix lifts $\tilde{f}$ for $f$ and $\tilde{g}$ for $g$. For $x\in X$, for any fixed $\tilde{x}\in p^{-1}(x)$, the number of sets $\Coin(\alpha\tilde{f}_i,\tilde{g}_j)$ containing $\tilde{x}$ is equal to the multiplicity of $x$.
\end{lem}

\begin{proof}
Let $|f(x)\cap g(x)|=r$, and fix $\tilde{x}\in p^{-1}(x)$. Given the fixed lifts $\tilde{f}$ and $\tilde{g}$, we can write \[
\begin{split}
f(x)&=\{q(\tilde{f}_1(\tilde{x})),\ldots,q(\tilde{f}_n(\tilde{x}))\} \\
g(x)&=\{q(\tilde{g}_1(\tilde{x})),\ldots,q(\tilde{g}_m(\tilde{x}))\}.
\end{split}
\]
Then there are indices $i_1,\ldots,i_r$ and $j_1,\ldots,j_r$ (pairwise distinct) such that $q(\tilde{f}_{i_\ell}(\tilde{x}))=q(\tilde{g}_{j_\ell}(\tilde{x}))$ for all $\ell\in \{1,\ldots,r\}$. For each $\ell$, it follows that there is an $\alpha_\ell\in \Delta$ such that $\alpha_\ell\tilde{f}_{i_\ell}(\tilde{x})=\tilde{g}_{i_\ell}(\tilde{x})$. Hence the number of sets $\Coin(\alpha\tilde{f}_i,\tilde{g}_j)$ containing $\tilde{x}$ is at least $r$.

To show that it is at most $r$, we show that each triple $(\alpha,i,j)$ for which $\tilde{x}\in \Coin(\alpha\tilde{f}_i,\tilde{g}_j)$, corresponds to a different point $q(\alpha\tilde{f}_i(\tilde{x}))=q(\tilde{g}_j(\tilde{x}))$ in $f(x)\cap g(x)$. Indeed, suppose two such triples $(\alpha,i,j)$ and $(\beta,k,\ell)$ correspond to the same point. From $q(\tilde{g}_j(\tilde{x}))=q(\tilde{g}_\ell(\tilde{x}))$ it follows that $j=\ell$; from \[
q(\tilde{f}_i(\tilde{x}))=q(\alpha\tilde{f}_i(\tilde{x}))=q(\beta\tilde{f}_k(\tilde{x}))=q(\tilde{f}_k(\tilde{x}))
\]
it follows that $i=k$; and then we have $\alpha\tilde{f}_i(\tilde{x})=\tilde{g}_j(\tilde{x})=\beta\tilde{f}_i(\tilde{x})$, from which it follows that $\alpha=\beta$.
\end{proof}

\begin{cor}\label{cor:multiple-coin}
For $x\in X$, the number of coincidence classes containing $x$ is at most equal to the multiplicity of $x$.
\end{cor}

\begin{proof}
Fix lifts $\tilde{f}$ for $f$ and $\tilde{g}$ for $g$, and fix $\tilde{x}\in p^{-1}(x)$. The number of coincidence classes containing $x$ is equal to $|A|$, where \[
A=\{[(\alpha\tilde{f}_i,\tilde{g}_j)] \mid x\in p\Coin(\alpha\tilde{f}_i,\tilde{g}_j) \}.
\]
Note that \[
\begin{split}
[(\beta\tilde{f}_k,\tilde{g}_\ell)]\in A &\iff \exists \gamma\in \Pi: \gamma\tilde{x}\in \Coin(\beta\tilde{f}_k,\tilde{g}_\ell) \\
&\iff \exists \gamma\in \Pi: \tilde{x}\in \Coin(\beta\tilde{f}_k\gamma,\tilde{g}_\ell\gamma) \\
&\iff \exists \gamma\in \Pi: [(\beta\tilde{f}_k\gamma,\tilde{g}_\ell\gamma)] \in \{[(\alpha\tilde{f}_i,\tilde{g}_j)] \mid \tilde{x}\in \Coin(\alpha\tilde{f}_i,\tilde{g}_j) \} \\
&\iff [(\beta\tilde{f}_k,\tilde{g}_\ell)] \in \{[(\alpha\tilde{f}_i,\tilde{g}_j)] \mid \tilde{x}\in \Coin(\alpha\tilde{f}_i,\tilde{g}_j) \}
\end{split}
\]
(in the last equivalence, we used that $(\beta\tilde{f}_k\gamma,\tilde{g}_\ell\gamma)\sim (\beta\tilde{f}_k,\tilde{g}_\ell)$ for all $\gamma\in \Pi$).
Hence, \[
\begin{split}
|A|&=|\{[(\alpha\tilde{f}_i,\tilde{g}_j)] \mid \tilde{x}\in \Coin(\alpha\tilde{f}_i,\tilde{g}_j) \}| \\
&\leq |\{(\alpha\tilde{f}_i,\tilde{g}_j) \mid \tilde{x}\in \Coin(\alpha\tilde{f}_i,\tilde{g}_j) \}| \\
&=|f(x)\cap g(x)|.\qedhere
\end{split}
\]
\end{proof}


The equality does not hold in general:

\begin{ex}\label{ex:f=g}
The equality does not hold in general: let $f=g$ be the square root map on the circle (see Figure \ref{fig:f21}), i.e.\ \[
\RR/\ZZ\multimap \RR/\ZZ: t+\ZZ\mapsto \{\tilde{f}_1(t)+\ZZ,\tilde{f}_2(t)+\ZZ \}
\] 
with $\tilde{f}_1(t)=\frac{t}{2}$ and $\tilde{f}_2(t)=\frac{t}{2}+\frac{1}{2}$. Then $\Coin(f,g)=\RR/\ZZ$, and all points are double coincidences. However, there is only one non-empty coincidence class. Indeed, if $\tilde{f}_i(t)+z=\tilde{f}_j(t)$ for $i,j\in \{1,2\}$ and $z\in \ZZ$, then either $i=j=1$ and $z=0$, or $i=j=2$ and $z=0$. The two lifting pairs $(\tilde{f}_1,\tilde{f}_1)$ and $(\tilde{f}_2,\tilde{f}_2)$ are equivalent, since $\tilde{f}_2(t)=\tilde{f}_1(t+1)$ for all $t\in \RR$.

\begin{figure}[h!]
\centering
\begin{tikzpicture}[scale=2.5]
\draw[thick,red] (0,0)--(1,0.5);
\draw[thick,red] (0,0.5)--(1,1);
\draw[thick,blue] (0,0)--(1,0.5);
\draw[thick,blue] (0,0.5)--(1,1);
\draw (0,0) rectangle (1,1);
\end{tikzpicture}
\caption{Graph of the square root map on the circle.}
\label{fig:f21}
\end{figure}
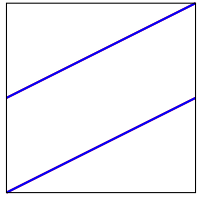
\end{ex}

\begin{rmk}
As we saw in the previous subsection, coincidence classes corresponding to different factors in a partition are always distinct. In particular, if one of the maps $f$ and $g$ is split (say $f$), each point in $f(x)\cap g(x)$ is a coincidence of $g$ with a different factor of $f$, hence corresponds to a different coincidence class containing $x$. So in that case, the equality in Corollary \ref{cor:multiple-coin} does hold.
\end{rmk}

\begin{rmk}\label{rmk:multiple-coin}
The equality in Corollary \ref{cor:multiple-coin} holds if and only if \[
|\{[(\alpha\tilde{f}_i,\tilde{g}_j)] \mid \tilde{x}\in \Coin(\alpha\tilde{f}_i,\tilde{g}_j) \}|
=|\{(\alpha\tilde{f}_i,\tilde{g}_j) \mid \tilde{x}\in \Coin(\alpha\tilde{f}_i,\tilde{g}_j) \}|
\]
for some (and equivalently, any) fixed $\tilde{x}\in p^{-1}(x)$. This means that, if $(\alpha,i,j)$ and $(\beta,k,\ell)$ are so that $\alpha\tilde{f}_i(\tilde{x})=\tilde{g}_j(\tilde{x})$ and $\beta\tilde{f}_k(\tilde{x})=\tilde{g}_\ell(\tilde{x})$ for some $\tilde{x}\in p^{-1}(x)$, then $(\alpha\tilde{f}_i,\tilde{g}_j)\sim (\beta\tilde{f}_k,\tilde{g}_\ell)$ only if $(\alpha,i,j)=(\beta,k,\ell)$.
\end{rmk}

\subsection{Affine maps of the circle}\label{subsec:cc-aff}

As far as maps of the circle are concerned, Example \ref{ex:f=g} is kind of an exception. Recall from \cite{brown2006} that any $n$-valued map on the circle is homotopic to an affine map \[
f_{n,d}:\RR/\ZZ\multimap \RR/\ZZ: t+\ZZ\mapsto \left\{\tfrac{d}{n}t+\ZZ, \tfrac{d}{n}t+\tfrac{1}{n}+\ZZ, \ldots, \tfrac{d}{n}t+\tfrac{n-1}{n}+\ZZ \right\}
\]
for some $d\in \ZZ$.

\begin{prop}\label{prop:aff-multiple}
For a pair of affine maps $(f,g)=(f_{n,d},f_{m,e})$ with $\frac{d}{n}\neq \frac{e}{m}$, all coincidence points have the same multiplicity $\gcd(n,m)$, and every coincidence point $x$ belongs to precisely $|f(x)\cap g(x)|$ classes.
\end{prop}

\begin{proof}
Let $r=\gcd(n,m)$, and write $n=ar$ and $m=br$. Note that a point $t+\ZZ\in \RR/\ZZ$ is a multiple coincidence point if and only if there exist multiple pairs $(k,\ell)\in \{0,\ldots,n-1\}\times\{0,\ldots,m-1\}$ such that \[
\frac{d}{n}t+\frac{k}{n} \equiv \frac{e}{m}t+\frac{\ell}{m} \mod \ZZ.
\]
Two pairs $(k,\ell)$ and $(k',\ell')$ satisfy this condition if and only if \[
\frac{dm-en}{nm}t \equiv \frac{\ell n-km}{nm} \mod \ZZ \qquad\text{and}\qquad \frac{dm-en}{nm}t \equiv \frac{\ell' n-k'm}{nm} \mod \ZZ.
\] 
Under the assumption $\frac{d}{n}\neq\frac{e}{m}$, this is satisfied for some $t\in \RR$ if and only if \[
(\ell-\ell')n \equiv (k-k')m \mod mn.
\] 
For $(\ell-\ell')n$ to be a multiple of $m$, we must have that $\ell-\ell'$ is a multiple of $b$, say $\ell-\ell'=bs$ with $s\in \ZZ$. It follows that $k-k'\equiv as \mod n$.
Taking into account that $k,k'\in \{0,\ldots,n-1\}$ and $\ell,\ell'\in \{0,\ldots,m-1\}$, all possible pairs are
\[
(k',\ell')\in \{(k,\ell),([k+a]_n,[\ell+b]_m),\ldots,([k+(r-1)a]_n,[\ell+(r-1)b]_m)\}
\] 
where we use $[\,\cdot\,]_n$ (resp.\ $[\,\cdot\,]_m$) to denote the remainder mod $n$ (resp.\ $m$).

Thus, each coincidence point is precisely $r$-fold. If $z\in \ZZ$ is such that \[
\frac{d}{n}t+\frac{k}{n}+z = \frac{e}{m}t+\frac{\ell}{m},
\]
then the sets $\Coin(\alpha\tilde{f}_i,\tilde{g}_j)$ containing $t$ are \[
\Coin(\tilde{f}_k,\tilde{g}_\ell),\enskip \Coin(\tilde{f}_{[k+a]_n},\tilde{g}_{[\ell+b]_m}),\enskip \ldots,\enskip \Coin(\tilde{f}_{[k+(r-1)a]_n},\tilde{g}_{[\ell+(r-1)b]_m})
\]
where we choose lifts $\tilde{f}$ and $\tilde{g}$ such that $\tilde{f}_{[k+sa]_n}(t)=\frac{d}{n}t+\frac{k+sa}{n}+z$ and $\tilde{g}_{[\ell+sb]_m}(t)=\frac{e}{m}t+\frac{\ell+sb}{m}$ for all $s\in \{0,\ldots,r-1\}$ (which we can do, since $[k+sa]_n$ is distinct for all these values of $s$, and similarly for $\ell+sb$).

Suppose two lifting pairs $(\tilde{f}_{[k+sa]_n},\tilde{g}_{[\ell+sb]_m})$ and $(\tilde{f}_{[k+s'a]_n},\tilde{g}_{[\ell+s'b]_m})$ are equivalent, for some $s,s'\in \{0,\ldots,r-1\}$. Then there exist $v,w\in \ZZ$ such that for all $t\in \RR$, \[
\begin{split}
&\begin{cases}
	\tilde{f}_{[k+sa]_n}(t+v)=\tilde{f}_{[k+s'a]_n}(t)+w \\
	\tilde{g}_{[\ell+sb]_m}(t+v)=\tilde{g}_{[\ell+s'b]_m}(t)+w
\end{cases}\\
&\iff
\begin{cases}
	\frac{d}{n}(t+v) + \frac{k+sa}{n} + z = \frac{d}{n}t + \frac{k+s'a}{n} + z + w \\
	\frac{e}{m}(t+v) + \frac{\ell+sb}{m} = \frac{e}{m}t + \frac{\ell+s'b}{m} + w.
\end{cases} \\
&\iff
\begin{cases}
	\frac{d}{n}v+\frac{(s-s')a}{n}=w \\
	\frac{e}{m}v+\frac{(s-s')b}{m}=w
\end{cases} \\
&\iff
\begin{cases}
	dv=n(w-\frac{s-s'}{r}) \\
	ev=m(w-\frac{s-s'}{r}).
\end{cases}
\end{split}
\]
If $s\neq s'$, then $\frac{s-s'}{r}$ cannot be equal to $w\in \ZZ$, since $s,s'\in \{0,\ldots,r-1\}$. Hence it follows that $\frac{d}{n}=\frac{e}{m}$.
\end{proof}

\begin{prop}\label{prop:aff-R}
For a pair of affine maps $(f,g)=(f_{n,d},f_{m,e})$ with $\frac{d}{n}\neq\frac{e}{m}$, there are $|dm-en|$ coincidence classes, each of which contains a single point.
\end{prop}

\begin{proof}
First suppose $f$ and $g$ are irreducible. Note that this is the case if and only if $\gcd(d,n)=1=\gcd(e,m)$. Let $\varphi=(\varphi_1,\ldots,\varphi_n;\sigma)$ and $\psi=(\psi_1,\ldots,\psi_n;\tau)$ be the morphisms induced by the lifts of $f$ and $g$ with factors $\tilde{f}_i(t)=\frac{d}{n}t+\frac{i}{n}$ and $\tilde{g}_j(t)=\frac{e}{m}t+\frac{j}{m}$. 
One easily verifies that, under the irreducibility assumption, $\sigma_1$ and $\tau_1$ in disjoint cycle notation are respectively an $n$-cycle and an $m$-cycle; and for all $i,j$ one has $S_{ij}=\frac{nm}{\gcd(n,m)}\ZZ$, $\varphi_i|_{S_{ij}}(z)=\frac{d}{n}z$ and $\psi_j|_{S_{ij}}(z)=\frac{e}{m}z$.
It follows that there are $\gcd(n,m)$ equivalence classes for the relation $\sim_{\sigma,\tau}$, and for all $i,j$ we have \[
\begin{split}
v\sim_{\varphi_i,\psi_j} w &\iff \exists z\in \tfrac{nm}{\gcd(n,m)}\ZZ: v=w+\tfrac{d}{n}z-\tfrac{e}{m}z \\
&\iff \exists z'\in \ZZ: v=w+\tfrac{dm-en}{nm}\tfrac{nm}{\gcd(n,m)}z' \\
&\iff v-w\in \tfrac{dm-en}{\gcd(n,m)}\ZZ.
\end{split}
\]
Under the assumption $\frac{d}{n}\neq \frac{e}{m}$, the number of classes for this relation is $\frac{|dm-en|}{\gcd(n,m)}$. Hence the total number of lifting classes is \[
\gcd(n,m)\cdot \frac{|dm-en|}{\gcd(n,m)} = |dm-en|.
\]

Now suppose $\gcd(d,n)=r$ and $\gcd(e,m)=s$. Write $d=rd_1$, $n=rn_1$, $e=se_1$ and $m=sm_1$. Then the decompositions into irreducibles are $f=\{f_0,\ldots,f_{r-1}\}$ and $g=\{g_0,\ldots,g_{s-1}\}$ with \[
\begin{split}
f_i(t+\ZZ)&=\left\{\tfrac{d_1}{n_1}t+\tfrac{i}{n}+\ZZ, \tfrac{d_1}{n_1}t+\tfrac{1}{n_1}+\tfrac{i}{n}+\ZZ, \ldots, \tfrac{d_1}{n_1}t+\tfrac{n_1-1}{n_1}+\tfrac{i}{n}+\ZZ \right\} \\
g_j(t+\ZZ)&=\left\{\tfrac{e_1}{m_1}t+\tfrac{j}{m}+\ZZ, \tfrac{e_1}{m_1}t+\tfrac{1}{m_1}+\tfrac{j}{m}+\ZZ, \ldots, \tfrac{e_1}{m_1}t+\tfrac{m_1-1}{m_1}+\tfrac{j}{m}+\ZZ \right\}
\end{split}
\]
for all $i,j$. Since $f_i\simeq f_0=f_{n_1,d_1}$ and $g_j\simeq g_0=f_{m_1,e_1}$ for all $i,j$, and since $\frac{d_1}{n_1}=\frac{d}{n}\neq \frac{e}{m}=\frac{e_1}{m_1}$, we have \[
\begin{split}
R(f,g)&=\sum_{i=0}^{r-1}\sum_{j=0}^{s-1} R(f_i,g_j) \\
&=rsR(f_{n_1,d_1},f_{m_1,e_1}) \\
&=|rs(d_1m_1-e_1n_1)| \\
&=|rd_1sm_1-se_1rn_1| \\
&=|dm-en|.
\end{split}
\]

To show that all coincidence classes contain a single point, we observe that for all $k,\ell,z$, the set $\Coin(\tilde{f}_k+z,\tilde{g}_\ell)$ consists of the single point \[
t=\frac{\ell n-km-nmz}{dm-en},
\]
hence $p\Coin(\tilde{f}_k+z,\tilde{g}_\ell)=\{t+\ZZ\}$ for this $t$.
\end{proof}


\begin{rmk}
In \cite[Theorem 5.1]{brownkolahi} it is claimed that the number of coincidence classes of $(f_{n,d},f_{m,e})$ is $\frac{|dm-en|}{\gcd(n,m)}$ rather than $|dm-en|$. This is because they do not take multiple coincidences into account; they show that there are $\frac{|dm-en|}{\gcd(n,m)}$ coincidence points, and no two points belong to the same coincidence class. But as points can belong to multiple coincidence classes, this does not imply their statement. Indeed, as we saw in Proposition \ref{prop:aff-multiple}, each of the $\frac{|dm-en|}{\gcd(n,m)}$ coincidence points belongs not to one but to $\gcd(n,m)$ distinct classes. 
\end{rmk}

\pagebreak

Note that this also gives an alternative proof of the second part of Proposition \ref{prop:aff-multiple}: since there are $|dm-en|$ distinct classes, and only $\frac{|dm-en|}{\gcd(n,m)}$ coincidence points, all $\gcd(n,m)$-fold, every coincidence point belongs to precisely as many classes as its multiplicity.

\begin{ex}\label{ex:f20-f21}
Consider the pair of affine maps $(f,g)=(f_{2,0},f_{2,1})$ on the circle (see Figure \ref{fig:f20-f21}). Explicitly, \[
\begin{split}
f:\RR/\ZZ \multimap \RR/\ZZ&: t+\ZZ \mapsto \{0+\ZZ,\tfrac{1}{2}+\ZZ \} \\
g:\RR/\ZZ \multimap \RR/\ZZ&: t+\ZZ \mapsto \{\tfrac{t}{2}+\ZZ,\tfrac{t}{2}+\tfrac{1}{2}+\ZZ \}.
\end{split}
\]
There is one double coincidence point $0+\ZZ$. To compute the coincidence classes, note that $(f,g)$ is the affine pair homotopic to the pair from Example \ref{ex:coin}; taking the lifts given by $\tilde{f}_1(t)=0$, $\tilde{f}_2(t)=\frac{1}{2}$, $\tilde{g}_1(t)=\tfrac{t}{2}$ and $\tilde{g}_2(t)=\tfrac{t}{2}+\frac{1}{2}$, one finds just as in Example \ref{ex:coin} that there are two lifting classes, $[(\tilde{f}_1,\tilde{g}_1)]=\{(\tilde{f}_1+z,\tilde{g}_i)\mid z,i\}$ and $[(\tilde{f}_2,\tilde{g}_2)]=\{(\tilde{f}_2+z,\tilde{g}_i)\mid z,i\}$. Hence the coincidence classes are \[
\begin{split}
p\Coin(\tilde{f}_1,\tilde{g}_1)&=\{0+\ZZ\} \\
p\Coin(\tilde{f}_2,\tilde{g}_2)&=\{0+\ZZ\}.
\end{split}
\]
So, as sets, the two coincidence classes are the same, but we take into account the lifting classes by which they are labeled so that they are really two different coincidence classes.

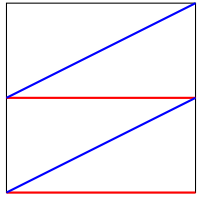
\begin{figure}[h!]
\begin{tikzpicture}[scale=2.5]
\draw (0,0) rectangle (1,1);
\draw[thick,red] (0,0)--(1,0);
\draw[thick,red] (0,0.5)--(1,0.5);
\draw[thick,blue] (0,0)--(1,0.5);
\draw[thick,blue] (0,0.5)--(1,1);
\end{tikzpicture}
\caption{Graphs of the affine maps $f=f_{2,0}$ (red) and $g=f_{2,1}$ (blue).}
\label{fig:f20-f21}
\end{figure}
\end{ex}

\begin{prop}\label{prop:aff-inessential}
For a pair of affine maps $(f,g)=(f_{n,d},f_{m,e})$ with $\frac{d}{n}=\frac{e}{m}$, there exists a homotopy $f\simeq f'$ such that $\Coin(f',g)=\varnothing$.
\end{prop}

\begin{proof}
Suppose without loss of generality that $e=cd$ and $m=cn$ for some $c\in \ZZ_{>0}$. Then we have \[
\begin{split}
f(t+\ZZ)&=\left\{\tfrac{d}{n}t+\ZZ, \tfrac{d}{n}t+\tfrac{1}{n}+\ZZ, \ldots, \tfrac{d}{n}t+\tfrac{n-1}{n}+\ZZ \right\} \\
g(t+\ZZ)&=\left\{\tfrac{d}{n}t+\ZZ, \tfrac{d}{n}t+\tfrac{1}{cn}+\ZZ, \ldots, \tfrac{d}{n}t+\tfrac{cn-1}{cn}+\ZZ \right\} \\
&=f(t+\ZZ)\cup (f(t+\ZZ)+\tfrac{1}{cn})\cup \ldots \cup (f(t+\ZZ)+\tfrac{c-1}{cn}).
\end{split}
\]
If we define $f'(t+\ZZ)=f(t+\ZZ)+\tfrac{1}{2cn}$ for all $t$, then $f'$ is homotopic to $f$ and $\Coin(f',g)=\varnothing$.
\end{proof}

Thus, at least for affine maps on the circle, all coincidence points $x$ that belong to less than $|f(x)\cap g(x)|$ classes can vanish after a homotopy. It therefore still seems reasonable to expect that the number of essential coincidence classes (those that cannot become empty after a homotopy) is equal to the minimal number of coincidences, counted with multiplicity, among all pairs homotopic to $(f,g)$.

\section{Essentiality and a Nielsen number}

Considering Example \ref{ex:f20-f21}, it is clear that we cannot generalise the usual definition of essentiality using a coincidence index. Recall that in classical Nielsen theory, an index $\ind(f,g,U)$ is defined for any admissible open $U\subseteq X$ (with $\Coin(f,g)\cap U$ compact), and the index of a coincidence class $C$ is defined as $\ind(f,g,U)$ with $U$ such that $\Coin(f,g)\cap U=C$. By default, this only takes into account what the coincidence class is as a subset of $X$; not the lifting class by which it is labeled. Defining an index like this in the situation above would force the indices of the two coincidence classes $p\Coin(\tilde{f}_1,\tilde{g}_1)$ and $p\Coin(\tilde{f}_2,\tilde{g}_2)$ to be equal, just because they are equal as sets.\footnote{In this specific example, one would actually expect the indices of these two coincidence classes to be equal. But one can easily come up with another situation in which this is not desirable; for instance, consider the maps on the circle defined by the following graphs:
\medskip
\begin{center}
\begin{tikzpicture}[scale=2]
	\draw (0,0) rectangle (1,1);
	\draw[thick,red] (0,0.18)--(1,0.18);
	\draw[thick,red] (0,0.8)--(1,0.8);
	\draw[thick,blue] (0,0)--(1,0.5);
	\draw[thick,blue] (0,0.5) to[out=50,in=180] (0.35,0.8) to[out=0,in=180] (0.6,0.7) to[out=0,in=225] (1,1);
	\draw[red] (0.1,0.275) node{\scriptsize$f$};
	\draw[blue] (0.7,0.425) node{\scriptsize$g$};
	
	\draw[dashed] (0.35,0)node[below]{\scriptsize$x$}--(0.35,0.8);
	\draw[dotted] (0.25,-0.125)--(0.25,1);
	\draw[dotted] (0.45,-0.125)--(0.45,1);
	\draw[decorate,decoration={brace,mirror}] (0.25,-0.125)--(0.45,-0.125) node[midway,below,yshift=-2pt]{\scriptsize$U$};
\end{tikzpicture}
\end{center}
The set $\Coin(f,g)\cap U$ consists of an isolated double coincidence point $x$ belonging to two classes. We would like $\ind(f,g,U)=\ind(f,g,x)$ to be zero for one of the classes, and non-zero for the other.
}

The more geometric approach of defining the Nielsen number simply as the number of coincidence classes that cannot vanish after a homotopy, does generalise to the $(n,m)$-valued case, hence we will take this as a definition for our Nielsen number.

\subsection{Coincidence classes and homotopies}\label{subsec:cc-hom}

We want to call a coincidence class (geometrically) essential if it cannot become empty after a homotopy. To make this precise, we need to be a bit more careful than in classical Nielsen theory: for a pair of homotopies $(F,G)$ starting at $(f,g)$, it is no longer true that a coincidence class of $(f,g)$ can be written as $C\cap (X\times \{0\})$ for a \emph{unique} coincidence class $C$ of $(F,G)$, so we cannot simply call it essential if, for any such $(F,G)$ with $C$ the class in question, $C\cap (X\times \{1\})$ is non-empty. However, we can make this work if we keep track of the lifting class by which the coincidence class is labeled.

Suppose $F:X\times I\to D_n(Y)$ and $G:X\times I\to D_m(Y)$ are homotopies. As in \cite{coin}, we will write \[
(F,G)=\{(f_t,g_t)\mid t\in I\}
\] 
with $f_t:X\to D_n(Y):x\mapsto F(x,t)$ and $g_t:X\to D_m(Y):x\mapsto G(x,t)$ for all $t$. We say $(F,G)$ `starts at' $(f,g)$ if $f_0=f$ and $g_0=g$.

The universal cover of $X\times I$ is $p\times \id:\tilde{X}\times I\to X\times I$, so lifts for $F$ and $G$ are maps making the following diagrams commute: \[
\begin{tikzcd}
\tilde{X}\times I \ar[r,"{(\tilde{F}_1,\ldots,\tilde{F}_n)}"] \ar[d,"p\times \id"'] &[2.3em] F_n(\tilde{Y},\Delta) \ar[d,"q^n"] & \tilde{X}\times I \ar[r,"{(\tilde{G}_1,\ldots,\tilde{G}_m)}"] \ar[d,"p\times \id"'] &[2.3em] F_m(\tilde{Y},\Delta) \ar[d,"q^m"] \\
X\times I \ar[r,"F"] & D_n(Y) & X\times I \ar[r,"G"] & D_m(Y).
\end{tikzcd}
\]
For all $t$, the map $((\tilde{f}_t)_1,\ldots,(\tilde{f}_t)_n)$ defined by $(\tilde{f}_t)_i(\tilde{x})=\tilde{F}_i(\tilde{x},t)$ is a lift of $f_t$, and the map $((\tilde{g}_t)_1,\ldots,(\tilde{g}_t)_m)$ defined by $(\tilde{g}_t)_j(\tilde{x})=\tilde{G}_j(\tilde{x},t)$ is a lift of $g_t$.

The covering group of $p\times \id$ is also $\Pi$, with action on $\tilde{X}\times I$ given by $\gamma(\tilde{x},t)=(\gamma\tilde{x},t)$. If $\varphi=(\varphi_1,\ldots,\varphi_n;\sigma):\Pi\to \Delta^n \rtimes \S_n$ is the morphism of covering groups induced by the lift $(\tilde{F}_1,\ldots,\tilde{F}_n)$ of $F$, then \[
(\tilde{f}_t)_i(\gamma\tilde{x})=\tilde{F}_i(\gamma\tilde{x},t)=\varphi_i(\gamma)\tilde{F}_{\sigma_\gamma^{-1}(i)}(\tilde{x},t)=\varphi_i(\gamma)(\tilde{f}_t)_{\sigma_\gamma^{-1}(i)}(\tilde{x})
\]
so $\varphi$ is also the morphism induced by the lift $((\tilde{f}_t)_1,\ldots,(\tilde{f}_t)_n)$ of $f_t$, for all $t$. Similarly, if $\psi=(\psi_1,\ldots,\psi_n;\tau):\Pi\to \Delta^m \rtimes \S_m$ is the morphism induced by the lift $(\tilde{G}_1,\ldots,\tilde{G}_m)$ of $G$, this is also the morphism induced by the lift $((\tilde{g}_t)_1,\ldots,(\tilde{g}_t)_m)$ of $g_t$, for all $t$. It follows that the set of coincidence classes of $(f_t,g_t)$ is \[
\{p\Coin(\alpha(\tilde{f}_t)_i,(\tilde{g}_t)_j) \mid [(i,j)]_{\sigma,\tau}; [\alpha]\in \R[\varphi_i,\psi_j] \}
\]
for all $t$.

Writing $C^{\alpha,i,j}=(p\times \id)\Coin(\alpha\tilde{F}_i,\tilde{G}_j)$ and $C^{\alpha,i,j}_t=p\Coin(\alpha(\tilde{f}_t)_i,(\tilde{g}_t)_j)$ for short, we have $C^{\alpha,i,j}_t\times\{t\}=C^{\alpha,i,j}\cap (X\times \{t\})$ for all $t\in I$. In particular, we have a bijection \begin{align*}
\xi_{(F,G)}:\{\text{coincidence classes of $(f,g)$}\} &\to \{\text{coincidence classes of $(F,G)$}\} \\
C=C^{\alpha,i,j}_0 &\mapsto \xi_{(F,G)}(C)\vcentcolon=C^{\alpha,i,j}.
\end{align*}

\begin{df}
We say a coincidence class $C$ of $(f,g)$ is \emph{geometrically essential} if it cannot become empty after a homotopy, i.e., there is no pair of homotopies $(F,G)$ starting at $(f,g)$ such that $\xi_{(F,G)}(C)\cap (X\times \{1\})=\varnothing$.
\end{df}

\begin{df}
We call the number of geometrically essential coincidence classes of $(f,g)$ the \emph{Nielsen number} $N(f,g)$.
\end{df}

\begin{thm}
$N(f,g)$ is a homotopy-invariant lower bound for the number of coincidences of $(f,g)$, counting multiplicity.
\end{thm}

\begin{proof}
Suppose $\{(f_t,g_t)\mid t\in I\}$ is a pair of homotopies. With the notation from above, we have a one-to-one correspondence \[
\begin{split}
\{\text{coincidence classes of $(f_0,g_0)$}\} &\leftrightarrow \{\text{coincidence classes of $(f_1,g_1)$}\} \\
C_0^{\alpha,i,j} &\leftrightarrow C_1^{\alpha,i,j}
\end{split}
\] 
If a coincidence class $C_1^{\alpha,i,j}$ of $(f_1,g_1)$ is inessential, there exists a pair of homotopies $\{(f_t,g_t) \mid t\in [1,2]\}$ such that, with the same notation, the coincidence class $C_2^{\alpha,i,j}$ of $(f_2,g_2)$ is empty. Then $(F,G)=\{(f_t,g_t) \mid t\in [0,2]\}$ is a pair of homotopies starting at $(f_0,g_0)$ such that $\xi_{(F,G)}(C_0^{\alpha,i,j})=C_2^{\alpha,i,j}$ is empty, so the class $C_0^{\alpha,i,j}$ of $(f_0,g_0)$ is inessential. Similarly, if $C_0^{\alpha,i,j}$ is inessential, then so is $C_1^{\alpha,i,j}$. Hence, $N(f,g)$ is homotopy-invariant.

By definition, essential coincidence classes are non-empty.
Each non-empty coincidence class contains at least one coincidence point. These points are not necessarily distinct for different coincidence classes, but since a point $x$ can belong to at most $|f(x)\cap g(x)|$ distinct coincidence classes, the number of coincidence points \emph{counted with multiplicity} must be an upper bound for the number of non-empty coincidence classes (hence for $N(f,g)$).
\end{proof}

In particular, we have \[
N(f,g)\leq \MC[f,g] \vcentcolon=\min\{\left|\Coin(f',g')\right|^* \mid f'\simeq f, g'\simeq g \}
\]
where $\left|\Coin(f',g')\right|^*$ is the number of coincidence points of $(f',g')$ counted with multiplicity.

Note that we have not yet shown that $N(f,g)$ is finite. This is less straightforward than in classical Nielsen theory; indeed, the usual proof goes as follows: the coincidence classes are disjoint and open in the compact set $\Coin(f,g)$, so there can only be finitely many non-empty classes (in particular, finitely many essential ones). This argument breaks down for $(n,m)$-valued pairs since coincidence classes are no longer disjoint (and they are not necessarily open, either). The result remains true, though; we will be able to show it with the tools from the following sections (see Theorem \ref{thm:N-finite}).

\subsection{Split lifts}\label{subsec:split-lift}

As for $(n,1)$-valued pairs (see \cite{coin}), we can understand coincidences of $(f,g)$ better by lifting $f$ and $g$ to a finite covering space of $X$ such that the lifts split into single-valued factors, to which we can apply single-valued coincidence theory.

\begin{df}
We call the smallest covering space $\hat{X}$ of $X$ such that the lifts of both $f$ and $g$ to $\hat{X}$ split, the \emph{splitting space} of the pair $(f,g)$.
\end{df}

We can give an explicit description of $\hat{X}$:

\begin{lem}\label{lem:splitting-space}
Let $\sigma:\Pi\to \S_n$ and $\tau:\Pi\to \S_m$ be the morphisms induced by some lifts $\tilde{f}$ of $f$ and $\tilde{g}$ of $g$, and $S_{ij}=\{\gamma\in \Pi \mid \sigma_\gamma(i)=i,\tau_\gamma(j)=j \}$ the corresponding subgroups of $\Pi$. Consider the subgroup $S=\cap_{i,j}S_{ij}=\ker\sigma\cap \ker\tau$ of $\Pi$. The space $\hat{X}=S\orb \tilde{X}$ is the splitting space of $(f,g)$.
\end{lem} 

\begin{proof}
By \cite[Theorem~3.1]{browngoncalves}, the lift $\hat{f}$ of $f$ to some space $\hat{X}=S\orb \tilde{X}$ is split if and only if $\hat{f}_*(S)$ is contained in $P_n(Y)$, where \[
\hat{f}_*:\pi_1(\hat{X})\cong S\to \pi_1(D_n(Y))\cong B_n(Y)
\] 
is the fundamental group morphism induced by $\hat{f}$; $B_n(Y)$ is the braid group on $n$ strands in $Y$ and $P_n(Y)$ the subgroup of pure braids. It is easy to see that this is the case if and only if $\ker\sigma\subseteq S$ (see \cite[Remarks~2.15 and~2.18]{coin}). Similarly, the lift $\hat{g}$ of $g$ to $\hat{X}$ is split if and only if $\ker\tau\subseteq S$. So the largest $S$ (corresponding to the smallest $\hat{X}=S\orb \tilde{X}$) such that both $\hat{f}$ and $\hat{g}$ split is $\ker\sigma\cap\ker\tau$.
\end{proof}

Note that
\begin{itemize}
	\item the space $\hat{X}$ is a finite regular covering space of $X$, since $S$ is a finite index normal subgroup of $\Pi$;
	\item as observed in \cite{coin}, the groups $\ker\sigma$ and $\ker\tau$ are independent of the chosen lifts for $f$ and $g$, hence so are $S$ and $\hat{X}$;
	\item we saw before that the induced morphisms $\sigma$ and $\tau$ are homotopy-invariant, hence so are $S$ and $\hat{X}$.
\end{itemize}

We write $\hat{p}:\hat{X}\to X$ for the covering map, and $\hat{f}\vcentcolon= f\circ \hat{p}:\hat{X}\to D_n(Y)$ and $\hat{g}\vcentcolon= g\circ \hat{p}:\hat{X}\to D_m(Y)$ for the split lifts of $f$ and $g$ to $\hat{X}$.
\[
\begin{tikzcd}
\hat{X} \ar[d,"\hat{p}"'] \ar[dr,"{\hat{f}=\{\hat{f}_1,\ldots,\hat{f}_n\}}"] &  & \hat{X} \ar[d,"\hat{p}"'] \ar[dr,"{\hat{g}=\{\hat{g}_1,\ldots,\hat{g}_m\}}"] & \\
X \ar[r,"f"{xshift=-1pt}] & D_n(Y) & X \ar[r,"g"{xshift=-2pt}] & D_m(Y)
\end{tikzcd}
\]
Let $\check{p}:\tilde{X}\to \hat{X}$ be the covering map such that $p=\hat{p}\circ \check{p}$.
In what follows, we always assume a fixed pair of lifts $(\tilde{f},\tilde{g})$ for $(f,g)$, and choose the order of the factors of $\hat{f}$ and $\hat{g}$ so that $\tilde{f}_i$ is a lift of $\hat{f}_i$ and $\tilde{g}_j$ is a lift of $\hat{g}_j$, for all $i$ and $j$.
\[
\begin{tikzcd}
\tilde{X} \ar[d,"\check{p}"'] \ar[r,"\tilde{f}_i"] & \tilde{Y} \ar[d,"q"]  &[2em] \tilde{X} \ar[d,"\check{p}"'] \ar[r,"\tilde{g}_j"] & \tilde{Y} \ar[d,"q"] \\
\hat{X} \ar[r,"\hat{f}_i"] & Y & \hat{X} \ar[r,"\hat{g}_j"] & Y
\end{tikzcd}
\]

\subsection{Coincidences of $(f,g)$ in terms of the split lifts}

The covering group of $\check{p}$ is $\Pi/S$. For $\gamma\in \Pi$, we let $\bar{\gamma}$ denote its image in $\Pi/S$. For all $\gamma\in \Pi$ and $\tilde{x}\in \tilde{X}$, we have \[
\check{p}(\gamma\tilde{x})=\bar{\gamma}\check{p}(\tilde{x}).
\]
If $\varphi=(\varphi_1,\ldots,\varphi_n;\sigma)$ and $\psi=(\psi_1,\ldots,\psi_n;\tau)$ are the morphisms induced by $\tilde{f}$ and $\tilde{g}$, then $\varphi_i|_S:S\to \Delta$ is the morphism induced by the lift $\tilde{f}_i$ of $\hat{f}_i$ and $\psi_j|_S:S\to \Delta$ is the morphism induced by the lift $\tilde{g}_j$ of $\hat{g}_j$, for all $i,j$. Moreover, if $\hat{x}=\check{p}(\tilde{x})$, then for all $\gamma\in \Pi$ and $i\in\{1,\ldots,n\}$, \[
\begin{split}
\hat{f}_i(\bar{\gamma}\hat{x})=q(\tilde{f}_i(\gamma\tilde{x}))=q(\varphi_i(\gamma)\tilde{f}_{\sigma^{-1}_\gamma(i)}(\tilde{x}))=q(\tilde{f}_{\sigma^{-1}_\gamma(i)}(\tilde{x}))=\hat{f}_{\sigma^{-1}_\gamma(i)}(\hat{x})
\end{split}
\]
and similarly, for all $j\in \{1,\ldots,m\}$, \[
\hat{g}_j(\bar{\gamma}\hat{x})=\hat{g}_{\tau^{-1}_\gamma(j)}(\hat{x}).
\]

\begin{prop}
We have \[
\Coin(f,g)=\bigcup_{i,j} \hat{p}\Coin(\hat{f}_i,\hat{g}_j)
\]
where
\begin{itemize}
\item[\rm(i)] if $(i,j)\sim_{\sigma,\tau}(k,\ell)$, then $\hat{p}\Coin(\hat{f}_i,\hat{g}_j)=\hat{p}\Coin(\hat{f}_k,\hat{g}_\ell)$;
\item[\rm(ii)] if $(i,j)\nsim_{\sigma,\tau}(k,\ell)$, then $\hat{p}\Coin(\hat{f}_i,\hat{g}_j)\cap \hat{p}\Coin(\hat{f}_k,\hat{g}_\ell)$ can only contain multiple coincidence points;
\item[\rm(iii)] for $x\in \Coin(f,g)$, the number of classes $[(i,j)]_{\sigma,\tau}$ such that $x\in \hat{p}\Coin(\hat{f}_i,\hat{g}_j)$ is at most equal to the multiplicity of $x$.
\end{itemize}
\end{prop}

\begin{proof}
Each $x\in X$ can be written as $\hat{p}(\hat{x})$ for some $\hat{x}\in \hat{X}$, and \[
\begin{split}
\hat{p}(\hat{x})\in \Coin(f,g) &\iff f(\hat{p}(\hat{x}))\cap g(\hat{p}(\hat{x}))\neq \varnothing \\
&\iff \{\hat{f}_1(\hat{x}),\ldots,\hat{f}_n(\hat{x})\}\cap \{\hat{g}_1(\hat{x}),\ldots,\hat{g}_m(\hat{x})\}\neq \varnothing \\
&\iff \exists i,j: \hat{f}_i(\hat{x})=\hat{g}_j(\hat{x}) \\
&\iff \exists i,j: \hat{p}(\hat{x})\in \hat{p}\Coin(\hat{f}_i,\hat{g}_j).
\end{split}
\]
For (i), suppose $(i,j)\sim_{\sigma,\tau}(k,\ell)$, and take $\gamma\in \Pi$ such that $k=\sigma_\gamma(i)$ and $\ell=\tau_\gamma(j)$. For $x\in \hat{p}\Coin(\hat{f}_i,\hat{g}_j)$, take $\hat{x}\in \hat{p}^{-1}(x)$ with $\hat{f}_i(\hat{x})=\hat{g}_j(\hat{x})$. Then $\bar{\gamma}\hat{x}\in \hat{p}^{-1}(x)$ satisfies \[
\hat{f}_k(\bar{\gamma}\hat{x})=\hat{f}_{\sigma^{-1}_\gamma(k)}(\hat{x})=\hat{f}_i(\hat{x})=\hat{g}_j(\hat{x})=\hat{g}_{\tau^{-1}_\gamma(\ell)}(\hat{x})=\hat{g}_\ell(\bar{\gamma}\hat{x}),
\]
so $x\in \hat{p}\Coin(\hat{f}_k,\hat{g}_\ell)$. Similarly, $\hat{p}\Coin(\hat{f}_k,\hat{g}_\ell)\subseteq \hat{p}\Coin(\hat{f}_i,\hat{g}_j)$.

Note that (ii) follows from (iii). For (iii), take a fixed $\hat{x}\in \hat{p}^{-1}(x)$. Since \[
f(x)\cap g(x)=\{\hat{f}_1(\hat{x}),\ldots,\hat{f}_n(\hat{x})\}\cap \{\hat{g}_1(\hat{x}),\ldots,\hat{g}_m(\hat{x})\},
\]
the number of pairs $(i,j)$ such that $\hat{x}\in \Coin(\hat{f}_i,\hat{g}_j)$ is exactly $|f(x)\cap g(x)|$. Now, if we write \[
A=\{[(i,j)]_{\sigma,\tau} \mid x\in \hat{p}\Coin(\hat{f}_i,\hat{g}_j)\},
\]
then for a fixed $\hat{x}\in \hat{p}^{-1}(x)$, \[
\begin{split}
[(k,\ell)]_{\sigma,\tau} \in A &\iff \exists \gamma\in \Pi: \bar{\gamma}\hat{x}\in \Coin(\hat{f}_k,\hat{g}_\ell) \\
&\iff \exists \gamma\in \Pi: \hat{x} \in \Coin(\hat{f}_{\sigma_\gamma^{-1}(k)},\hat{g}_{\tau_\gamma^{-1}(\ell)}) \\
&\iff \exists \gamma\in \Pi: [(\sigma_\gamma^{-1}(k),\tau_\gamma^{-1}(\ell))]_{\sigma,\tau} \in \{[(i,j)]_{\sigma,\tau} \mid \hat{x} \in \Coin(\hat{f}_i,\hat{g}_j) \} \\
&\iff [(k,\ell)]_{\sigma,\tau} \in \{[(i,j)]_{\sigma,\tau} \mid \hat{x} \in \Coin(\hat{f}_i,\hat{g}_j) \},
\end{split}
\]
so \[
\begin{split}
|A| &= |\{[(i,j)]_{\sigma,\tau} \mid \hat{x} \in \Coin(\hat{f}_i,\hat{g}_j) \}| \\
&\leq |\{(i,j) \mid \hat{x} \in \Coin(\hat{f}_i,\hat{g}_j) \}| \\
&=|f(x)\cap g(x)|. \qedhere
\end{split}
\]
\end{proof}

\subsection{Coincidence classes in terms of the split lifts}

For each $i$ and $j$, the coincidence classes of $(\hat{f}_i,\hat{g}_j)$ are $\check{p}\Coin(\alpha\tilde{f}_i,\tilde{g}_j)$, indexed by the Reidemeister classes $[\alpha]\in \R[\varphi_i|_S,\psi_j|_S]$, i.e.\ the equivalence classes of $\Delta$ for the relation \[
\alpha\sim_{\varphi_i|_S,\psi_j|_S} \beta \iff \exists \gamma\in S: \alpha=\psi_j(\gamma^{-1})\beta\varphi_i(\gamma).
\]
Let us abbreviate $\varphi_i|_S=\vcentcolon\varphi'_i$ and $\psi_j|_S=\vcentcolon\psi'_j$.

\begin{lem}\label{lem:R-inverse-image}
For a fixed $[\alpha]\in \R[\varphi_i,\psi_j]$, the number of classes $[\beta]\in \R[\varphi'_i,\psi'_j]$ such that $\beta\sim_{\varphi_i,\psi_j}\alpha$ is equal to \[
\frac{[S_{ij}:S]}{|u_{ij}(\coin(\mu_\alpha \varphi_i,\psi_j))|},
\]
where $\mu_\alpha:\Delta\to\Delta:\delta\mapsto \alpha\delta\alpha^{-1}$ is the conjugation morphism; \[
\coin(\mu_\alpha \varphi_i,\psi_j)=\{\gamma\in S_{ij} \mid \alpha\varphi_i(\gamma)\alpha^{-1}=\psi_j(\gamma) \},
\]
and $u_{ij}:S_{ij}\to S_{ij}/S$ is the natural projection.
\end{lem}

\begin{proof}
For clarity, we will use $[\,\cdot\,]_{\varphi_i,\psi_j}$ to denote classes in $\R[\varphi_i,\psi_j]$ and $[\,\cdot\,]_{\varphi'_i,\psi'_j}$ for classes in $\R[\varphi'_i,\psi'_j]$.

For $[\alpha]_{\varphi_i,\psi_j}\in \R[\varphi_i,\psi_j]$, consider the action of $S_{ij}/S$ on the set \[
\{[\beta]_{\varphi'_i,\psi'_j}\in \R[\varphi'_i,\psi'_j] \mid \beta\sim_{\varphi_i,\psi_j}\alpha \}
\]
given by \[
\bar{\gamma}\cdot [\beta]_{\varphi'_i,\psi'_j} = [\psi_j(\gamma^{-1})\beta\varphi_i(\gamma)]_{\varphi'_i,\psi'_j}.
\]
It is easily verified that this action is well-defined and transitive, and that the stabilizer of $[\alpha]_{\varphi'_i,\psi'_j}$ is $u_{ij}(\coin(\mu_\alpha \varphi_i,\psi_j))$. The result follows from the orbit-stabilizer theorem.
\end{proof}

\begin{thm}
We have
\[
R(f,g)\geq \frac{1}{[\Pi:S]}\sum_{i=1}^n\sum_{j=1}^m R(\hat{f}_i,\hat{g}_j)
\]
with an equality if and only if $|u_{ij}(\coin(\mu_\alpha \varphi_i,\psi_j))|=1$ for all $\alpha,i,j$, or one of the terms $R(\hat{f}_i,\hat{g}_j)$ is infinite.
\end{thm}

\enlargethispage{7mm}

\begin{proof}
We have \[
\begin{split}
R(f,g) &= \sum_{[(i,j)]_{\sigma,\tau}} \sum_{[\alpha]\in \R[\varphi_i,\psi_j]} 1 \\
&= \sum_{i=1}^n\sum_{j=1}^m \frac{1}{[\Pi:S_{ij}]} \sum_{[\alpha]\in \R[\varphi_i,\psi_j]} 1 \\
&= \sum_{i=1}^n\sum_{j=1}^m \frac{1}{[\Pi:S_{ij}]} \sum_{[\alpha]\in \R[\varphi'_i,\psi'_j]} \frac{|u_{ij}(\coin(\mu_\alpha \varphi_i,\psi_j))|}{[S_{ij}:S]} \\
&= \frac{1}{[\Pi:S]} \sum_{i=1}^n\sum_{j=1}^m \sum_{[\alpha]\in \R[\varphi'_i,\psi'_j]} |u_{ij}(\coin(\mu_\alpha \varphi_i,\psi_j))|.
\end{split}
\]
The second equality holds since the number of elements of $[(i,j)]_{\sigma,\tau}$ is $[\Pi:S_{ij}]$ (by the orbit-stabilizer theorem), and the content of the sum is independent of the chosen representative $(i,j)$. The third equality follows from Lemma \ref{lem:R-inverse-image}.

\pagebreak

Note that $|u_{ij}(\coin(\mu_\alpha \varphi_i,\psi_j))|\geq 1$ for all $\alpha,i,j$, since $\coin(\mu_\alpha \varphi_i,\psi_j)$ always contains the trivial element. On the other hand, for all $i,j$, \[
R(\hat{f}_i,\hat{g}_j)=\sum_{[\alpha]\in \R[\varphi'_i,\psi'_j]} 1.
\]

If $R(\hat{f}_i,\hat{g}_j)=\infty$ for some $i,j$, then $\R[\varphi'_i,\psi'_j]$ is infinite, so also $R(f,g)=\infty$.
If all sets $\R[\varphi'_i,\psi'_j]$ are finite, then \[
\begin{split}
\frac{1}{[\Pi:S]} \sum_{i=1}^n\sum_{j=1}^m \sum_{[\alpha]\in \R[\varphi'_i,\psi'_j]} |u_{ij}(\coin(\mu_\alpha \varphi_i,\psi_j))| &\geq \frac{1}{[\Pi:S]} \sum_{i=1}^n\sum_{j=1}^m \sum_{[\alpha]\in \R[\varphi'_i,\psi'_j]} 1 \\
&= \frac{1}{[\Pi:S]} \sum_{i=1}^n\sum_{j=1}^m R(\hat{f}_i,\hat{g}_j)
\end{split}
\]
where the equality holds if and only if $|u_{ij}(\coin(\mu_\alpha \varphi_i,\psi_j))|=1$ for all $\alpha,i,j$.
\end{proof}

Using the coincidence classes of $(\hat{f}_i,\hat{g}_j)$, we can now easily prove:

\begin{thm}\label{thm:N-finite}
The number of non-empty coincidence classes of $(f,g)$ is finite. In particular, the Nielsen number $N(f,g)$ is finite.
\end{thm}

\begin{proof}
Note that the coincidence classes of $(f,g)$ can be written as \[
p\Coin(\alpha\tilde{f}_i,\tilde{g}_j)=\hat{p}(\check{p}\Coin(\alpha\tilde{f}_i,\tilde{g}_j)),
\]
where $\check{p}\Coin(\alpha\tilde{f}_i,\tilde{g}_j)$ are the coincidence classes of $(\hat{f}_i,\hat{g}_j)$.
In particular, the set of coincidence classes of $(f,g)$ is \[
\{\hat{p}(C) \mid \text{$C$ is a coincidence class of $(\hat{f}_i,\hat{g}_j)$ for some $i,j$} \}.
\]
Here $\hat{p}(C)$ is non-empty only if $C$ is non-empty, and for every $i$ and $j$ there are only finitely many non-empty coincidence classes of $(\hat{f}_i,\hat{g}_j)$. Thus, the set of non-empty coincidence classes of $(f,g)$ is finite. In particular, the number of essential coincidence classes, which are by definition non-empty, is finite.
\end{proof}

Another result relating the coincidence classes of $(f,g)$ to those of $(\hat{f}_i,\hat{g}_j)$ that we would like to generalise from the $(n,1)$-valued case, is that for all $(\alpha,i,j)$, the restriction \[
\hat{p}:\check{p}\Coin(\alpha\tilde{f}_i,\tilde{g}_j)\to p\Coin(\alpha\tilde{f}_i,\tilde{g}_j)
\]
is a $|u_{ij}(\coin(\mu_\alpha \varphi_i,\psi_j))|$-fold covering map. But this is no longer true in general:

\begin{ex}
Consider again the case where $f=g$ is the square root map on the circle (i.e.\ the affine map $f_{2,1}$; see Example \ref{ex:f=g}). The split lift of $f$ is \[
\hat{f}:\RR/2\ZZ \multimap \RR/\ZZ: t+2\ZZ\mapsto \{\tfrac{t}{2}+\ZZ,\,\tfrac{t}{2}+\tfrac{1}{2}+\ZZ \}.
\]
Then \[
\hat{p}:\check{p}\Coin(\tilde{f}_1,\tilde{f}_1)=\RR/2\ZZ\to p\Coin(\tilde{f}_1,\tilde{f}_1)=\RR/\ZZ:t+2\ZZ\mapsto t+\ZZ
\]
is a double cover, whereas \[
|u_{11}(\coin(\varphi_1,\varphi_1))|=[S_{11}:S]=[2\ZZ:2\ZZ]=1.
\]
\end{ex}

But this was a special example, where the number of coincidence classes containing a point $x$ was strictly less than $|f(x)\cap g(x)|$. If this is not the case, our desired result does hold:

\begin{lem}
For all $(\alpha,i,j)$, if $x\in p\Coin(\alpha\tilde{f}_i,\tilde{g}_j)$ belongs to exactly $|f(x)\cap g(x)|$ coincidence classes, then \[
|\hat{p}^{-1}(x)\cap\check{p}\Coin(\alpha\tilde{f}_i,\tilde{g}_j)|=|u_{ij}(\coin(\mu_\alpha \varphi_i,\psi_j))|.
\]
\end{lem}

\begin{proof}
Fix $\hat{x}\in \hat{p}^{-1}(x)\cap \check{p}\Coin(\alpha\tilde{f}_i,\tilde{g}_j)$, and take $\tilde{x}\in \check{p}^{-1}(\hat{x})\cap \Coin(\alpha\tilde{f}_i,\tilde{g}_j)$. We will describe all other elements of $\hat{p}^{-1}(x)\cap \check{p}\Coin(\alpha\tilde{f}_i,\tilde{g}_j)$ in terms of $\hat{x}$, and show that there are $|u_{ij}(\coin(\mu_\alpha \varphi_i,\psi_j))|$ of them.

Any other element of $\hat{p}^{-1}(x)$ can be written as $\bar{\delta}\hat{x}$ for some $\bar{\delta}\in \Pi/S$. If $\bar{\delta}\hat{x}\in \check{p}\Coin(\alpha\tilde{f}_i,\tilde{g}_j)$, take $\tilde{x}'\in \check{p}^{-1}(\bar{\delta}\hat{x})\cap \Coin(\alpha\tilde{f}_i,\tilde{g}_j)$. Since $p(\tilde{x}')=x=p(\tilde{x})$, there is a $\gamma\in \Pi$ such that $\tilde{x}'=\gamma\tilde{x}$; since $\bar{\delta}\hat{x}=\check{p}(\gamma\tilde{x})=\bar{\gamma}\hat{x}$, it follows that $\bar{\delta}=\bar{\gamma}$. Now, $\tilde{x}\in \Coin(\alpha\tilde{f}_i,\tilde{g}_j)$ means \[
\alpha\tilde{f}_i(\tilde{x})=\tilde{g}_j(\tilde{x})
\]
and $\gamma\tilde{x}\in \Coin(\alpha\tilde{f}_i,\tilde{g}_j)$ means \[
\begin{split}
\alpha\tilde{f}_i(\gamma\tilde{x})=\tilde{g}_j(\gamma\tilde{x}) &\iff \alpha\varphi_i(\gamma)\tilde{f}_{\sigma_\gamma^{-1}(i)}(\tilde{x})=\psi_j(\gamma)\tilde{g}_{\tau_\gamma^{-1}(j)}(\tilde{x}) \\ &\iff \psi_j(\gamma)^{-1}\alpha\varphi_i(\gamma)\tilde{f}_{\sigma_\gamma^{-1}(i)}(\tilde{x})=\tilde{g}_{\tau_\gamma^{-1}(j)}(\tilde{x}).
\end{split}
\]
By by Remark \ref{rmk:multiple-coin}, it follows that $(\alpha,i,j)=(\psi_j(\gamma)^{-1}\alpha\varphi_i(\gamma),\sigma_\gamma^{-1}(i),\tau_\gamma^{-1}(j))$. In other words, $i=\sigma_\gamma^{-1}(i)$ and $j=\tau_\gamma^{-1}(j)$, which means $\gamma\in S_{ij}$; and \[
\alpha=\psi_j(\gamma)^{-1}\alpha\varphi_i(\gamma) \iff \alpha\varphi_i(\gamma)\alpha^{-1}=\psi_j(\gamma),
\]
so together, $\gamma\in \coin(\mu_\alpha \varphi_i,\psi_j)$. Conversely, if $\gamma\in \coin(\mu_{\alpha} \varphi_i,\psi_j)$, then \[
\begin{split}
	\alpha\tilde{f}_i(\gamma\tilde{x})&=\alpha\varphi_i(\gamma)\tilde{f}_i(\tilde{x})\\&=\mu_\alpha\varphi_i(\gamma)\alpha\tilde{f}_i(\tilde{x})\\&=\mu_\alpha\varphi_i(\gamma)\tilde{g}_j(\tilde{x})\\&=\psi_j(\gamma)\tilde{g}_j(\tilde{x})\\&=\tilde{g}_j(\gamma\tilde{x});
\end{split}
\] 
that is, $\gamma\tilde{x}\in \Coin(\alpha\tilde{f}_i,\tilde{g}_j)$; and hence $\bar{\gamma}\hat{x}=\check{p}(\gamma\tilde{x})\in \hat{p}^{-1}(x)\cap\check{p}\Coin(\alpha\tilde{f}_i,\tilde{g}_j)$.
Thus, we have \[
\hat{p}^{-1}(x)\cap\check{p}\Coin(\alpha\tilde{f}_i,\tilde{g}_j)=\{\bar{\gamma}\hat{x} \mid \bar{\gamma}\in u_{ij}(\coin(\mu_\alpha \varphi_i,\psi_j)) \},
\]
whose cardinality is $|u_{ij}(\coin(\mu_\alpha \varphi_i,\psi_j))|$.
\end{proof}

\begin{ex}
For the pair $(f,g)=(f_{2,0},f_{2,1})$ from Example \ref{ex:f20-f21}, the split pair of lifts $(\hat{f},\hat{g})$ is given by \[
\begin{split}
\hat{f}:\RR/2\ZZ \multimap \RR/\ZZ&: t+2\ZZ \mapsto \{\hat{f}_1(t+2\ZZ)=0+\ZZ,\,\hat{f}_2(t+2\ZZ)=\tfrac{1}{2}+\ZZ \} \\
\hat{g}:\RR/2\ZZ \multimap \RR/\ZZ&: t+2\ZZ \mapsto \{\hat{g}_1(t+2\ZZ)=\tfrac{t}{2}+\ZZ,\,\hat{g}_2(t+2\ZZ)=\tfrac{t}{2}+\tfrac{1}{2}+\ZZ \}.
\end{split}
\]
Each pair $(\hat{f}_i,\hat{g}_j)$ has a single coincidence class: \begin{align*}
(\hat{f}_1,\hat{g}_1)&:\enskip \check{p}\Coin(\tilde{f}_1,\tilde{g}_1)=\{0+2\ZZ\} \\
(\hat{f}_1,\hat{g}_2)&:\enskip \check{p}\Coin(\tilde{f}_1,\tilde{g}_2)=\{1+2\ZZ\} \\
(\hat{f}_2,\hat{g}_1)&:\enskip \check{p}\Coin(\tilde{f}_2,\tilde{g}_1)=\{1+2\ZZ\} \\
(\hat{f}_2,\hat{g}_2)&:\enskip \check{p}\Coin(\tilde{f}_2,\tilde{g}_2)=\{0+2\ZZ\}.
\end{align*}
Under $\hat{p}$, the first two classes project down to the coincidence class $p\Coin(\tilde{f}_1,\tilde{g}_1)$; the second two to $p\Coin(\tilde{f}_2,\tilde{g}_2)$ (since $(\tilde{f}_1,\tilde{g}_1)\sim (\tilde{f}_1,\tilde{g}_2)$ and $(\tilde{f}_2,\tilde{g}_1)\sim (\tilde{f}_2,\tilde{g}_2)$ as lifting pairs for $(f,g)$).

Note that all sets $S_{ij}$ are equal to $S=2\ZZ$. Hence $|u_{ij}(\coin(\mu_\alpha\varphi_i,\psi_j))|=1$ for all $\alpha,i,j$. Correspondingly, for every $i$ and $j$ the map $\hat{p}:\check{p}\Coin(\tilde{f}_i,\tilde{g}_j)\to p\Coin(\tilde{f}_i,\tilde{g}_j)$ is a $1$-fold cover.
\end{ex}

\subsection{Essentiality in terms of the split lifts}

\begin{thm}\label{thm:ind-ess}
Suppose a coincidence index is defined for maps $\hat{X}\to Y$ (e.g.\ when both spaces are compact orientable manifolds of equal dimension). Then whenever \linebreak $\ind(\hat{f}_i,\hat{g}_j,\check{p}\Coin(\alpha\tilde{f}_i,\tilde{g}_j))\neq 0$, the coincidence class $p\Coin(\alpha\tilde{f}_i,\tilde{g}_j)$ of $(f,g)$ is geometrically essential.
\end{thm}

\begin{proof}
Let $\{(f_t,g_t)\mid t\in I\}$ be a pair of homotopies with $f_0=f$ and $g_0=g$. Define the maps $(\tilde{f}_t)_i,(\tilde{g}_t)_j:\tilde{X}\to\tilde{Y}$ as in section \ref{subsec:cc-hom}. Also define $\hat{f}_t=f_t\circ \hat{p}$ and $\hat{g}_t=g_t\circ \hat{p}$ for all $t\in I$; then $\{(\hat{f}_t,\hat{g}_t)\mid t\in I\}$ is a homotopy with $\hat{f}_0=\hat{f}$ and $\hat{g}_0=\hat{g}$. For all $t$, order the factors $(\hat{f}_t)_1,\ldots,(\hat{f}_t)_n$ so that $(\tilde{f}_t)_i$ is a lift of $(\hat{f}_t)_i$ for all $i$, and similarly for $g$. 

Now suppose $p\Coin(\alpha(\tilde{f}_1)_i,(\tilde{g}_1)_j)=\varnothing$. Then also $\check{p}\Coin(\alpha(\tilde{f}_1)_i,(\tilde{g}_1)_j)=\varnothing$, hence \[
\ind((\hat{f}_1)_i,(\hat{g}_1)_j,\check{p}\Coin(\alpha(\tilde{f}_1)_i,(\tilde{g}_1)_j))=0.
\]
But then \[
\begin{split}
\ind(\hat{f}_i,\hat{g}_j,\check{p}\Coin(\alpha\tilde{f}_i,\tilde{g}_j))&=\ind((\hat{f}_0)_i,(\hat{g}_0)_j,\check{p}\Coin(\alpha(\tilde{f}_0)_i,(\tilde{g}_0)_j))\\&=\ind((\hat{f}_1)_i,(\hat{g}_1)_j,\check{p}\Coin(\alpha(\tilde{f}_1)_i,(\tilde{g}_1)_j))=0.
\end{split}
\]
(The second equality is single-valued coincidence theory, but is also shown explicitly with this notation for $(n,1)$-valued pairs in the proof of \cite[Theorem 2.31]{coin}; in particular this holds for $n=1$.)
\end{proof}

\begin{cor}\label{cor:av-N}
If a coincidence index is defined for maps $\hat{X}\to Y$, we have \[
N(f,g)\geq \frac{1}{[\Pi:S]} \sum_{i=1}^n \sum_{j=1}^m N(\hat{f}_i,\hat{g}_j),
\]
where $N(\hat{f}_i,\hat{g}_j)$ is the usual single-valued Nielsen coincidence number, i.e.\ the number of coincidence classes with non-zero index.
\end{cor}

\begin{proof}
Define \[
\eps_{\alpha,i,j}= \begin{cases}
1 & \text{if $p\Coin(\alpha\tilde{f}_i,\tilde{g}_j)$ is an essential coincidence class of $(f,g)$} \\
0 & \text{otherwise}
\end{cases}
\]
and \[
\eps'_{\alpha,i,j}= \begin{cases}
1 & \text{if $\ind(\hat{f}_i,\hat{g}_j,\check{p}\Coin(\alpha\tilde{f}_i,\tilde{g}_j))\neq 0$} \\
0 & \text{otherwise}.
\end{cases}
\]
By Theorem \ref{thm:ind-ess}, we have $\eps'_{\alpha,i,j}\leq\eps_{\alpha,i,j}$ for all $\alpha,i,j$. Hence
\begin{align*}
N(f,g) &= \sum_{[(i,j)]_{\sigma,\tau}} \sum_{[\alpha]\in \R[\varphi_i,\psi_j]} \eps_{\alpha,i,j} \\
&= \sum_{i=1}^n\sum_{j=1}^m \frac{1}{[\Pi:S_{ij}]} \sum_{[\alpha]\in \R[\varphi_i,\psi_j]} \eps_{\alpha,i,j} \\
&= \sum_{i=1}^n\sum_{j=1}^m \frac{1}{[\Pi:S_{ij}]} \sum_{[\alpha]\in \R[\varphi'_i,\psi'_j]} \frac{|u_{ij}(\coin(\mu_\alpha \varphi_i,\psi_j))|}{[S_{ij}:S]}\,\eps_{\alpha,i,j} \\
&\geq \frac{1}{[\Pi:S]} \sum_{i=1}^n\sum_{j=1}^m \sum_{[\alpha]\in \R[\varphi'_i,\psi'_j]} \eps'_{\alpha,i,j} \\
&= \frac{1}{[\Pi:S]} \sum_{i=1}^n\sum_{j=1}^m N(\hat{f}_i,\hat{g}_j). \qedhere
\end{align*}
\end{proof}

\begin{rmk}
If a coincidence index is defined for maps $\hat{X}\to Y$, one could define the \emph{algebraic Nielsen number} $N_a(f,g)$ as the number of coincidence classes $p\Coin(\alpha\tilde{f}_i,\tilde{g}_j)$ of $(f,g)$ such that \[
\ind(\hat{f}_i,\hat{g}_j,\check{p}\Coin(\alpha\tilde{f}_i,\tilde{g}_j))\neq 0.
\]
By the proof of Corollary \ref{cor:av-N}, this Nielsen number satisfies \[
N(f,g) \geq N_a(f,g) \geq \frac{1}{[\Pi:S]} \sum_{i=1}^n \sum_{j=1}^m N(\hat{f}_i,\hat{g}_j)
\]
where the second inequality is an equality if and only if $|u_{ij}(\coin(\mu_\alpha \varphi_i,\psi_j))|=1$ for all $\alpha,i,j$ for which $\ind(\hat{f}_i,\hat{g}_j,\check{p}\Coin(\alpha\tilde{f}_i,\tilde{g}_j))\neq 0$.
\end{rmk}

\subsection{Affine maps on the circle}


\begin{lem}\label{lem:aff-ind}
If $\frac{d}{n}\neq\frac{e}{m}$, then all coincidence classes of the affine pair $(f_{n,d},f_{m,e})$ are essential.
\end{lem}

\begin{proof}
We can assume $f_{n,d}$ and $f_{m,e}$ are irreducible; otherwise, each coincidence class is a coincidence class of a pair of irreducible factors $(f_i,g_j)$, and if it is essential as a coincidence class of $(f_i,g_j)$, it is certainly essential for $(f,g)$.

In the irreducible case, we have $S=\frac{nm}{\gcd(n,m)}\ZZ$, and the factors of the split lifts are given by $\hat{f}_i(t+\frac{nm}{\gcd(n,m)}\ZZ)=\frac{d}{n}t+\frac{i}{n}+\ZZ$ and $\hat{g}_j(t+\frac{nm}{\gcd(n,m)}\ZZ)=\frac{e}{m}t+\frac{j}{m}+\ZZ$ for all $i,j$. Let $\tilde{f}$ and $\tilde{g}$ be the corresponding lifts given by $\tilde{f}_i(t)=\frac{d}{n}t+\frac{i}{n}$ and $\tilde{g}_j(t)=\frac{e}{m}t+\frac{j}{m}$. Since each coincidence class $\check{p}\Coin(\tilde{f}_i+z,\tilde{g}_j)$ consists of one isolated coincidence point, and $\tilde{f}_i+z$ and $\tilde{g}_j$ are differentiable maps on $\RR$ whose difference has non-zero derivative $\frac{e}{m}-\frac{d}{n}$ at every point, it follows from single-valued coincidence theory that \[
\ind(\hat{f}_i,\hat{g}_j,\check{p}\Coin(\tilde{f}_i+z,\tilde{g}_j))=\text{sgn}(\tfrac{e}{m}-\tfrac{d}{n})\neq 0
\]
for all $i,j$ and $z\in \ZZ$.
\end{proof}

\begin{thm}
For any pair of affine maps $(f,g)=(f_{n,d},f_{m,e})$ on the circle, \[
N(f,g)=|dm-en|.
\]
\end{thm}

\begin{proof}
If $\frac{d}{n}=\frac{e}{m}$, it follows from Proposition \ref{prop:aff-inessential} that all coincidence classes are inessential, hence $N(f,g)=0=|dm-en|$. If $\frac{d}{n}\neq\frac{e}{m}$, we saw in Proposition \ref{prop:aff-R} that there are $|dm-en|$ coincidence classes, and we just proved that they are all essential.
\end{proof}

Since any $(n,m)$-valued pair $(f,g)$ on the circle is homotopic to an affine pair $(f_{n,d},f_{m,e})$, with $d=\deg(f)$ and $e=\deg(g)$,\footnote{The degree of an $n$-valued circle map $f$ is defined in \cite{brown2006} as the integer $d$ such that $f$ is homotopic to $f_{n,d}$, but it can also be computed by considering the morphisms $\varphi_i:S_i\to \ZZ$ induced by $f$ (which are the same for all $i$, and independent of the chosen lift); they are given by $S_i=n\ZZ\to \ZZ:nz\mapsto dz$.} we obtain:

\begin{cor}
For any $(n,m)$-valued pair $(f,g)$ on the circle, \[
N(f,g)=|dm-en|
\]
where $d=\deg(f)$ and $e=\deg(g)$.
\end{cor}

\begin{cor}
For $(n,m)$-valued pairs on the circle, \[
N(f,g)=\MC[f,g].
\]
\end{cor}

\begin{proof}
Let $d=\deg(f)$ and $e=\deg(g)$. If $\frac{d}{n}\neq\frac{e}{m}$, then $(f,g)$ is homotopic to the affine pair $(f_{n,d},f_{m,e})$, for which the number of coincidence points (counted with multiplicity) is exactly $|dm-en|$. If $\frac{d}{n}=\frac{e}{m}$, then $(f,g)$ is homotopic to the affine pair $(f_{n,d},f_{m,e})$, which is in turn homotopic to a coincidence-free pair (see Proposition \ref{prop:aff-inessential}).
\end{proof}

\section{Graph coincidences}

In \cite{graffstaeckerzeromski}, the authors propose an alternative coincidence theory for $(n,m)$-valued pairs by looking at \emph{graph coincidences} of $(f,g)$, i.e.\ points $(x,y)\in X\times Y$ such that $y\in f(x)\cap g(x)$. They define graph coincidence classes geometrically, by calling two graph coincidence points $(x,y)$ and $(x',y')$ equivalent if there is a path $c:I\to X$ from $x$ to $x'$ such that, if $fc=\{f_1,\ldots,f_n\}$ and $gc=\{g_1,\ldots,g_m\}$, there exist $i$ and $j$ such that $f_i$ and $g_j$ are homotopic paths from $y$ to $y'$. They call a graph coincidence class $C_0$ of $(f,g)$ geometrically inessential if there exists a pair of homotopies $(F,G)$ starting at $(f,g)$ such that the graph coincidence class $C$ of $(F,G)$ containing $C_0$ (which they show exists and is unique) satisfies $C\cap (X\times \{1\}\times Y)=\varnothing$. The number of essential graph coincidence classes is denoted $N(f:g)$.

\begin{prop}
$N(f:g)=N(f,g)$ for every $(n,m)$-valued pair $f,g:X\multimap Y$ between finite polyhedra.
\end{prop}

\begin{proof}
First of all, note that $x\in \Coin(f,g)$ if and only if there is a $y\in Y$ such that $(x,y)$ is a graph coincidence of $(f,g)$.
If $(x,y)$ and $(x',y')$ belong to the same graph coincidence class, then it follows immediately from Proposition \ref{prop:htpy} that $x$ and $x'$ belong to the same coincidence class. Conversely, if $x$ and $x'$ belong to the coincidence class $p\Coin(\tilde{f}_i,\tilde{g}_j)$, for some lift factors $\tilde{f}_i$ of $f$ and $\tilde{g}_j$ of $g$, pick $\tilde{x}\in p^{-1}(x)\cap \Coin(\tilde{f}_i,\tilde{g}_j)$ and $\tilde{x}'\in p^{-1}(x')\cap \Coin(\tilde{f}_i,\tilde{g}_j)$, write $\tilde{f}_i(\tilde{x})=\tilde{g}_j(\tilde{x})=\tilde{y}$ and $\tilde{f}_i(\tilde{x}')=\tilde{g}_j(\tilde{x}')=\tilde{y}'$, and let $y=q(\tilde{y})$, $y'=q(\tilde{y}')$. Then the points $(x,y)$ and $(x',y')$ belong to the same graph coincidence class. Indeed, let $\tilde{c}:I\to \tilde{X}$ be a path from $\tilde{x}$ to $\tilde{x}'$. Since $\tilde{f}_i\tilde{c}$ and $\tilde{g}_j\tilde{c}$ have the same begin and endpoint (respectively $\tilde{y}$ and $\tilde{y}'$), and $\tilde{Y}$ is simply connected, they are path homotopic, hence so are the factors $f_i=q\tilde{f}_i\tilde{c}$ and $g_j=q\tilde{g}_j\tilde{c}$ of $fc$ and $gc$, which are paths from $y$ to $y'$.

We obtain a one-to-one correspondence \begin{align*}
\{\text{coincidence classes of $(f,g)$}\} &\leftrightarrow \{\text{graph coincidence classes of $(f,g)$}\} \\
p\Coin(\tilde{f}_i,\tilde{g}_j) & \mapsto \{(p(\tilde{x}),q(\tilde{y})) \mid \tilde{x}\in \Coin(\tilde{f}_i,\tilde{g}_j),\,\tilde{f}_i(\tilde{x})=\tilde{g}_j(\tilde{x})=\tilde{y} \} \\
 \pr_1(C) &\mapsfrom C
\end{align*}
where $\pr_1:X\times Y\to X$ is the projection onto the first component.

To compare essentiality of the classes, recall that a graph class $C_0$ of $(f,g)$ is inessential if, for some pair of homotopies $(F,G)$ starting at $(f,g)$, the graph class $C$ of $(F,G)$ containing $C_0$ satisfies $C\cap (X\times \{1\}\times Y)=\varnothing$. We must show that this is the case if and only if $\pr_1(C_0)$ is inessential in our sense, that is, $\xi_{(F,G)}(\pr_1(C_0))\cap(X\times \{1\})=\varnothing$ for some pair of homotopies $(F,G)$ -- we show it is satisfied for the same pair; in fact $\xi_{(F,G)}(\pr_1(C_0))=\pr_1(C)$.
(Here we abusively denote the projection $X\times I\times Y\to X\times I$ by $\pr_1$ as well.)

Indeed, if $\pr_1(C_0)=p\Coin(\tilde{f}_i,\tilde{g}_j)$ under the bijection above, then \[
C_0=\{(p(\tilde{x}),q(\tilde{y})) \mid \tilde{x}\in \Coin(\tilde{f}_i,\tilde{g}_j),\,\tilde{f}_i(\tilde{x})=\tilde{g}_j(\tilde{x})=\tilde{y} \},
\] and the graph coincidence class of $(F,G)$ containing this set is \[
C=\{((p\times\id)(\tilde{x},t),q(\tilde{y})) \mid (\tilde{x},t)\in \Coin(\tilde{F}_i,\tilde{G}_j),\,\tilde{F}_i(\tilde{x},t)=\tilde{G}_j(\tilde{x},t)=\tilde{y} \}
\]
for which $\pr_1(C)=(p\times \id)\Coin(\tilde{F}_i,\tilde{G}_j)$ is exactly $\xi_{(F,G)}(\pr_1(C_0))$.
\end{proof}

For $(n,m)$-valued pairs of self-maps of the circle, the authors of \cite{graffstaeckerzeromski} define the index of a graph coincidence class as follows. To each pair $(f,g)$, they associate a pair of \emph{simple torus loops} $\ell_f,\ell_g:I\to T^2$, whose images are precisely the graphs of $f$ and $g$, respectively. If $(x,y)$ is a graph coincidence point of $(f,g)$, then there is a (unique) pair $(t_1,t_2)\in I\times I$ such that $\ell_f(t_1)=\ell_g(t_2)=(x,y)$. For pairs of maps $(f_1:X_1\to Y,f_2:X_2\to Y)$ between compact orientable manifolds with $\dim(X_1)+\dim(X_2)=\dim(Y)$, there exists a Nielsen-type theory called \emph{McCord intersection theory}, concerned with the \emph{intersection points} $(x_1,x_2)\in X_1\times X_2$ such that $f_1(x_1)=f_2(x_2)$. The intersection points are partitioned into intersection classes, and to each class one can associate a McCord index. The authors of \cite{graffstaeckerzeromski} establish a one-to-one correspondence between the graph coincidence classes of $(f,g)$ and the intersection classes of the pair of associated torus loops $(\ell_f,\ell_g)$. They define the index of the coincidence class as the McCord index of the corresponding intersection class.

\begin{prop}
For an $(n,m)$-valued pair $(f,g)$ on the circle, the index of a graph coincidence class $C$ is equal to our index $\ind_1(\hat{f}_i,\hat{g}_j,\check{p}\Coin(\alpha\tilde{f}_i,\tilde{g}_j))$, where $p\Coin(\alpha\tilde{f}_i,\tilde{g}_j)$ is the coincidence class corresponding to $C$.
\end{prop}

\begin{proof}
It is shown in \cite{graffstaeckerzeromski} that all graph coincidence classes of an affine $(n,m)$-valued pair $(f,g)=(f_{n,d},f_{m,e})$ on the circle have the same index, namely \[
\text{sgn}(ne-md).
\]
As shown in Lemma \ref{lem:aff-ind} for $\frac{d}{n}\neq \frac{e}{m}$, and as follows from the fact that $N(f,g)=0$ if $\frac{d}{n}=\frac{e}{m}$, this is also the index of every class $p\Coin(\alpha\tilde{f}_i,\tilde{g}_j)$ of $(\hat{f}_i,\hat{g}_j)$. Since the index of a coincidence class is a homotopy-invariant (both $\ind_1$ and the McCord intersection index), the result follows for arbitrary maps.
\end{proof}


\begin{thebibliography}{1}

\bibitem{brown2006}
R.~F. Brown.
\newblock {Fixed points of $n$-valued multimaps of the circle}.
\newblock {\em Bulletin of the Polish Academy of Sciences, Mathematics},
54(2):153--162, 2006.

\bibitem{charlotte1}
R.~F. Brown, C.~Deconinck, K.~Dekimpe, and P.~C. Staecker.
\newblock {Lifting classes for the fixed point theory of $n$-valued maps}.
\newblock {\em Topology and its Applications}, 274:107125, 2020.

\bibitem{browngoncalves}
R.~F. Brown and D.~L. Gon\c{c}alves.
\newblock {On the topology of $n$-valued maps}.
\newblock {\em Advanced Fixed Point Theory}, 8(2):205--220, 2018.

\bibitem{brownkolahi}
R.~F. Brown and K.~Kolahi.
\newblock {Nielsen coincidence, fixed point and root theories of $n$-valued
maps}.
\newblock {\em Journal of Fixed Point Theory and Applications}, 14:309--324,
2013.

\bibitem{coin}
K.~Dekimpe and L.~De~Weerdt.
\newblock {Nielsen coincidence theory for $(n,1)$-valued pairs}.
\newblock {\em Topological Methods in Nonlinear Analysis}, accepted, 2026.

\bibitem{graffstaeckerzeromski}
G.~Graff, P.~C. Staecker, and A.~Zeromski.
\newblock {Nielsen coincidence theory of $(n,m)$-valued pairs of maps}.
\newblock {\em ArXiv preprint: 2605.07015v1}, 2026.

\bibitem{staecker2021}
P.~C. Staecker.
\newblock {Partitions of $n$-valued maps}.
\newblock {\em ArXiv preprint: 2101.09326}, 2021.

\end{thebibliography}
\end{document}